\documentclass[review,onefignum,onetabnum]{siamart171218}

\usepackage{enumerate}
\usepackage{multirow}

\usepackage{lipsum}
\usepackage{amsfonts}
\usepackage{graphicx}

\usepackage{float}
\usepackage{subfigure}

\usepackage{epstopdf}
\usepackage{algorithmic}

\ifpdf
  \DeclareGraphicsExtensions{.eps,.pdf,.png,.jpg}
\else
  \DeclareGraphicsExtensions{.eps}
\fi

\newsiamremark{remark}{Remark}
\newsiamremark{hypothesis}{Hypothesis}
\crefname{hypothesis}{Hypothesis}{Hypotheses}
\newsiamthm{claim}{Claim}

\headers{Shape optimization in Stokes flows with phase-field approaches}{F. Li and J. Yang}

\title{A provably efficient monotonic-decreasing algorithm for shape optimization in Stokes flows by phase-field approaches}

\author{Futuan Li\thanks{Department of Mathematics, Southern University of Science and Technology, Shenzhen 518055, China,
  \email{11930689@mail.sustech.edu.cn}.}
\and Jiang Yang\thanks{SUSTech International Center for Mathematics \&Guangdong Provincial Key Laboratory of Computational Science and Material Design, Southern University of Science and Technology, Shenzhen 518055, China,  \email{yangj7@sustech.edu.cn}.}
 }

\usepackage{amsopn}

\makeatletter
\newcommand*{\addFileDependency}[1]{
  \typeout{(#1)}
  \@addtofilelist{#1}
  \IfFileExists{#1}{}{\typeout{No file #1.}}
}
\makeatother

\begin{document}

\maketitle

\begin{abstract}
In this work, we study shape optimization problems in the Stokes flows. By phase-field approaches, the resulted total objective function consists of the dissipation energy of the fluids and the Ginzburg--Landau energy functional as a regularizing term for the generated diffusive interface, together with Lagrangian multiplayer for volume constraint. An efficient decoupled scheme is proposed to implement by the gradient flow approach to decrease the objective function. In each loop, we first update the velocity field by solving the Stokes equation with the phase field variable given in the previous iteration, which is followed by updating the phase field variable by solving an Allen--Cahn-type equation using a stabilized scheme. We then take a cut-off post-processing for the phase-field variable to constrain its value in $[0,1]$. In the last step of each loop, the Lagrangian parameter is updated with an appropriate artificial time step. We rigorously prove that the proposed scheme permits an unconditionally monotonic-decreasing property, which allows us to use the adaptive mesh strategy. To enhance the overall efficiency of the algorithm, in each loop we update the phase field variable and Lagrangian parameter several time steps but update the velocity field only one time. Numerical results for various shape optimizations are presented to validate the effectiveness of our numerical scheme. 
\end{abstract}



\section{Introduction}
\label{sec:introduction}
Shape optimization in fluid mechanics has attracted extensive attentions in recent years \cite{Sokolowski92Introduction,Mohammadi01Applied,Haslinger03Introduction,Guest06Topology,Borrvall03Topology}.
It has been widely applied spanning from the optimizations of transport vehicles like airplanes and cars \cite{Hazra08Multigrid}, biomechanical and industrial production processes \cite{Abraham05Shape,Zhang15Topology}, to the optimization of pipe bending \cite{Gersborg-Hansen05Topology}. 

Different approaches have been proposed to numerically solve the shape optimization problem, among which shape sensitivity analysis is usually used in the field of shape optimization in fluids. The level set method has gained plenty of attention to numerical solutions to partial differential equations \cite{Osher88Fronts,Wang03A,Zhu10Variational,Allaire04Structural,Allarie14Shape}, and it is also a popular way for shape optimization in fluid \cite{Jiang16The,Zhang15Topology} without an unfavorable procedure of re-meshing. The target shape is represented as a zero contour of the level set function, and the optimized shape is archived by updating the level set function via the Hamilton-Jacobi equation. However, re-initialization is always required to keep the regularity of the level set function. 
Recently, the threshold dynamics method \cite{ChenAn,Esedoglu06Threshold,Wang17An} has been applied to shape optimization for fluids. The authors proposed an iterative scheme that has a monotonic-decaying property.

The phase-field method is an alternative approach to the level set method for various free boundary problems \cite{Sun07Sharp,Shen09An}, shape and topology optimization for elastic structures \cite{Blank16Sharp,Takezawa10Shape,Wallin12Optimal} and photonic devices \cite{Takezawa14Phase,Wu18A}. It is always served as an easy-to-use methodology for numerical simulations of the phase transition phenomenon \cite{Gaginalp86An}, as well as the two-phase flow simulations \cite{Jacqmin99Calculation,Liu03A,Yue05Diffuse,Zhou103D}. Bourdin and Chambolle \cite{Bourdin03Design-dependent} firstly adopted the phase-field model to shape and topology optimization and investigated the feasibility of the phase-field model. In the approach, the phase-field function is used to keep track of the motion of the interface. Differing from the level set method, the phase-field method stands out in its treatment of the interface as a physically diffuse thin layer, where the interface is sharp conceptually but regularized numerically by a continuous function between $[0,1]$. Due to the generation of the interface, there is usually mixing energy governing the dynamics of the diffuse interface. Wallin \cite{Wallin12Optimal} and Blank \cite{Blank14Relating} use the Ginzburg-Landau energy as the mixing energy over the interfacial layer, but the mixing energy can be dated back to Van der Waals energy. In the framework of the phase-field model, the Ginzburg-Landau energy is considered as a regularization item of the objective functional, including the cost of intermediate material densities and the cost of creating surfaces. To preserve the existence of the proposed phase-field model, a Moreau--Yosida regularized version for shape optimization is introduced in \cite{Garcke15Numerical}, which leads to the deficiency of monotonic-delaying property. Indeed, there have been plenty of research jobs on the simulations \cite{Takezawa10Shape,Wallin12Optimal,Blank14Relating,Garcke15Numerical,Garcke17A,Li19A} for shape optimization in the past few decades. But to the best of our knowledge, under the constraint $\phi\in[0,1]$, there is still no decoupled monotonic-decaying scheme for the phase-field model based on shape optimization of incompressible flows up to now.

Stimulated by the pioneering work \cite{Garcke15Numerical,Li19A,ChenAn,Li20Arbitrarily}, we propose an efficient decoupled monotonic-decaying phase-field scheme for shape optimization and constrain the phase-field values in $[0,1]$ by a cut-off postprocessing. It is an iterative method via updating the velocity variable, the phase variable, and the volume parameter. In detail, To preserve the monotonic-decaying property, three key points need to deal with in every iterative loop: 
\begin{itemize}
	\item The penalized term $\alpha(\phi)$ for objective energy is designed to be a quadratic function, which forces itself not negative.
	\item Reset $\phi = 1$ if $\phi >1$ and set $\phi=0$ if $\phi < 0$ after solving the corresponding phase-field equation, which constraints its value in $[0,1]$.
	\item The Lagrange parameter $\lambda$ for the volume constrained condition is updated by a variable step to satisfy the monotonic-decaying property after a cut-off postprocessing for the phase-field function.
\end{itemize}

Therefore, the phase-field model can be mainly solved in an alternative way: a Stokes equation, a time-dependent Allen-Cahn-type equation, a cut-off postprocessing, and an evolution equation for the constraint of a given volume. In the time discretization for the Allen-Cahn-type equation, there have been plenty of linear unconditionally energy stable schemes for the phase-field equation to be researched recently \cite{Baskaran13Convergence,Guan14A,Shen10Numerical,Zhao17Numerical,Shen19A,Shen18Convergence,Shen18The}. In this paper, we merely focus on the total monotonic-decaying scheme for shape optimization. A semi-implicit stabilized scheme for shape optimization is proposed, and we prove that it is energy stable.
In space discretization, we adopt the common $P2-P1$ finite element method to solve the Stokes equation and the $P1$ finite element to solve the phase-field equation. In our implementation, projection and interpolation techniques are used to treat the mismatched solutions of the phase-field equation and the Stokes equation.

The paper is organized as follows. In \cref{sec:MP}, we describe the phase-field model for shape optimization. In \cref{sec:ENS}, we propose a decoupled, energy monotonic-decaying scheme. In \cref{sec:FED}, we give a detail on the finite element approximation of Stokes equation and phase-field equation, which is compatible with the energy monotonic-decaying scheme of \cref{sec:ENS}. In \cref{sec:Numerical-Simu}, various numerical results are presented to validate the effectiveness of the proposed schemes. The conclusions follow in
\cref{sec:conclusions}.

\section{Model problem}
\label{sec:MP}
In this section, we will introduce the phase-field model for shape optimization in fluids. The aim of shape optimization is to find an ultimate optimal shape under the constraints of Stokes equation and a given volume.
\subsection{The original model}
Let $\mathbf{u}$ and $p$ be the velocity and the pressure of incompressible flows, respectively, which are defined in a bounded Lipschitz domain $\Omega$. 

A model for shape optimization in Stokes flows can be described as follows \cite{Garcke15Numerical}:
\begin{equation} \label{eq:original-model}
 \min_{\Omega} J_0(\Omega):=\int_{\Omega} \frac{1}{2} \left|\nabla \mathbf{u}\right|^2 d \mathbf{x}  + \eta P_{\Omega},
\end{equation}
subject to
\begin{equation} \label{eq:original-model-volume}
\left|\Omega\right|=\beta V_0,
\end{equation}
and
\begin{equation} \label{eq:stoke-eq}
\left\{
\begin{aligned}
& - \Delta \mathbf{u} + \nabla p  = 0,   &\mathbf{x} \in \Omega,\\
& \nabla \cdot \mathbf{u} = 0, & \mathbf{x} \in \Omega,\\
& \mathbf{u} = \mathbf{g}, & \mathbf{x} \in \Gamma_{in}, \\
& \mathbf{u} = \mathbf{0}, & \mathbf{x} \in \Gamma_{wall}, \\
& \frac{\partial \mathbf{u}}{\partial \mathbf{n}} + p \mathbf{n} = \mathbf{0}, &  \mathbf{x} \in \Gamma_{out}.
\end{aligned}
\right.
\end{equation}
where the parameter $\eta$ is a positive penalized constant, and the symbol $P_{\Omega}$ denotes the length of the perimeter of the free boundary. The first term of the objective functional $J_0(\Omega)$ denotes the compliance energy in the fluid \cite{Borrvall03Topology,Garcke15Numerical}.

\begin{figure}[H]
	\centering
	\includegraphics[width=0.6\textwidth]{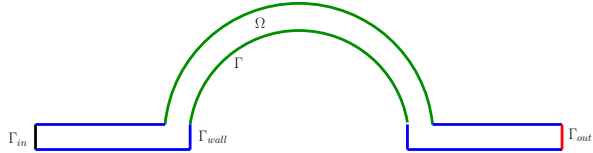} 
	\caption{The domain $\Omega$ for shape optimization with different boundaries: the inflow boundary $\left(\Gamma_{in}\right)$, the wall boundary $\left(\Gamma_{wall}\right)$, the outflow boundary $\left(\Gamma_{out}\right)$, and the free boundary $\left(\Gamma\right)$.}
	\label{fig:domain-omega}
\end{figure} 

The boundary conditions of the Stokes equation are specified (e.g. Fig.\ref{fig:domain-omega}). The boundary $\partial \Omega$ consists of four parts: $\partial \Omega:=\Gamma_{in} \cup \Gamma_{wall} \cup \Gamma_{out} \cup \Gamma$. The inflow boundary $\Gamma_{in}$, the wall boundary $\Gamma_{wall}$, and the outflow boundary $\Gamma_{out}$ are fixed during the optimized procedure of shape optimization, while the free boundary $\Gamma$ is the concerning part as an optimized shape. There is a natural Newman boundary condition on $\Gamma_{out}$ while Dirichlet boundary conditions are shown on $\Gamma_{in}\cap \Gamma_{wall}$.

\subsection{Phase-field method}
\label{sec:Mat}
The idea of the phase-field method is to assume that the interface is a diffuse thin layer, as plotted in the left figure of Figure \ref{fig:design-concept}. In this sense, the interface is sharp conceptually but regularized numerically. The phase-field function $\phi$ is defined over the whole design domain $D$ $\left(\in \mathbb{R}^d,\ d=2,3\right)$, which consists of three different parts, the fluid domain $\Omega$ with $\phi=1$, the fictitious diffuse interface $\xi$ with $0<\phi<1$, and the void domain $D \backslash (\Omega \cup \xi)$ with $\phi=0$, as plotted in the right figure of Figure \ref{fig:design-concept}.  
\begin{figure}[H]
	\centering
	\includegraphics[width=0.19\textwidth]{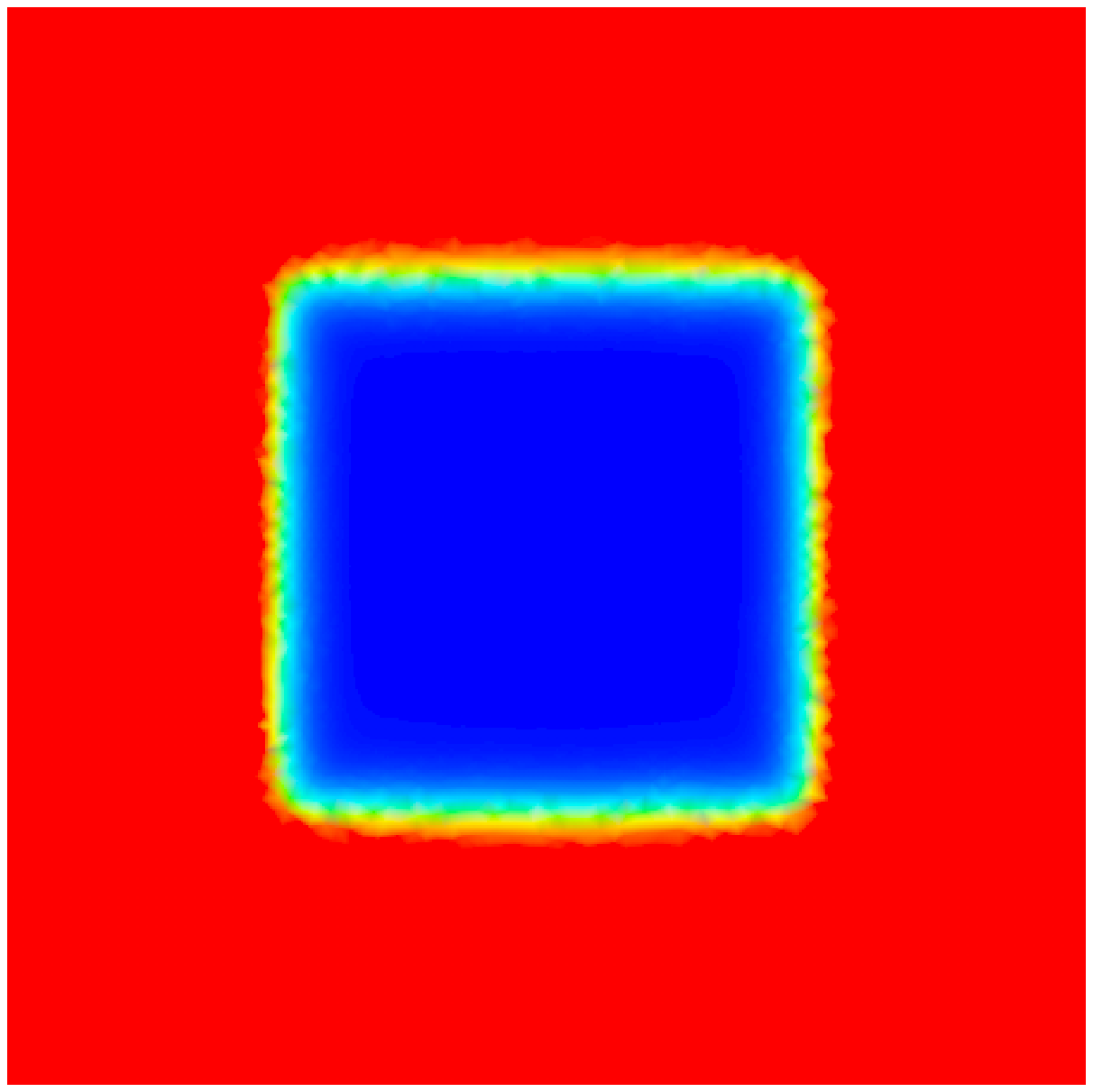}   \hspace{2.5cm}
	\includegraphics[width=0.26\textwidth]{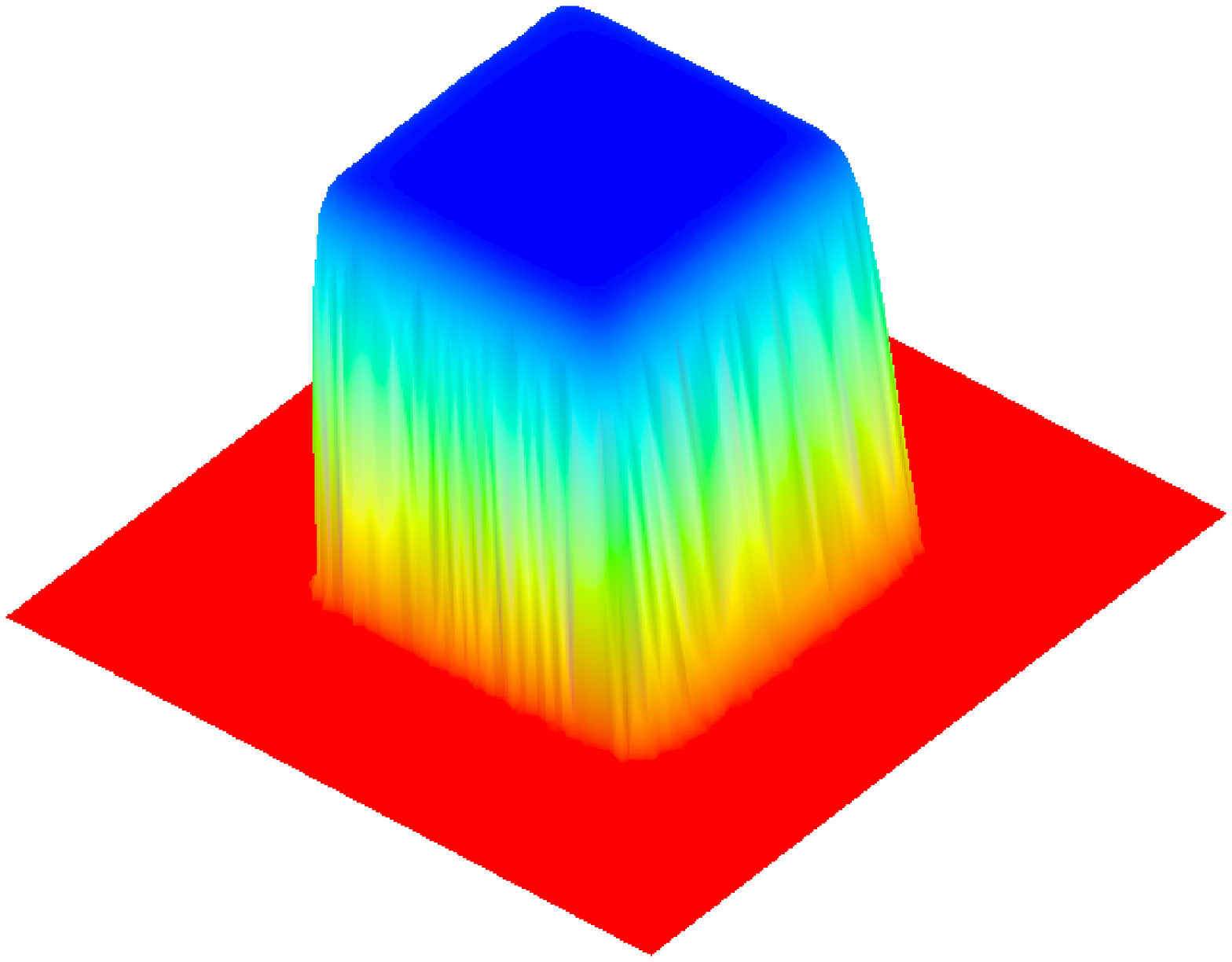}
	\caption{Concept of shape optimization: the domain (left) and the phase-field function (right) in 2D}
	\label{fig:design-concept}
\end{figure} 

\subsubsection{The objective functional}
Assume that $D$ is a designed domain that contains the optimized shape domain $\Omega$. 
To ensure that the velocity vanishes on the inner boundary in the fluid region $D$, an additional term is required to add into the objective functional to penalize the velocity field in the void domain, which is expressed as (see, for instant, \cite{Garcke15Numerical,ChenAn,Borrvall03Topology,Zhang15Topology,Abraham05Shape,Li19Shape,Yan17Shape}). 
\begin{equation} \label{eq:J2u}
J_{\alpha}(\phi,\mathbf{u}) = \int_D \frac{1}{2}  \left|\nabla \mathbf{u}\right|^2 d \mathbf{x} + \int_D \frac{1}{2}\alpha(\phi)\left|\mathbf{u}\right|^2 d \mathbf{x},
\end{equation}
where $\alpha(\phi)=\alpha_0(1-\phi)^2$ is a penalized function in this paper, which projects the total domain $D$ to the local domain $(D\ \backslash \Omega) \cup \xi $. 

On the other hand, the mixing energy, to penalize the generated diffuse interface, $J_{\epsilon}(\phi)$ replacing the length of perimeter is added to the objective function \cite{Wallin12Optimal,Garcke15Numerical,Shen10Numerical}, such as,
\begin{equation} \label{eq:J3u}
J_{\epsilon}(\phi) =  \int_D \frac{\epsilon}{2} \left|\nabla \phi\right|^2 + \frac{1}{ \epsilon} F(\phi) d \mathbf{x},
\end{equation}
where $F(\phi):=\frac{1}{4} \phi^2 (\phi-1)^2$ and $F^{\prime}(\phi):=f(\phi)=\phi(1-\phi)(1/2 -\phi)$. $\epsilon$ is the width that scales with the thickness of the diffuse interface. Apparently, $J_{\epsilon}(\phi)$ is only contributed by the interface which is proportional to the length of the interface.

An additional penalization term $\alpha(\phi)\mathbf{u}$ is added into the Stokes equation \cite{Garcke15Numerical}, which interpolates between the Stokes flows $\left(\alpha(\phi)=0\right)$ in $\left\{ \mathbf{x}\in D \ | \ \phi(\mathbf{x})=1\right\}$ and some Darcy flows ($\alpha(\phi)=+\infty$) through porous medium with permeability in void domain with $\phi=0$. The penalized state equation is presented directly in the following formulations:
\begin{equation} \label{eq:state-equation-phi}
\left\{
\begin{aligned}
& - \Delta \mathbf{u} + \nabla p + \alpha(\phi) \mathbf{u} = 0, \quad  & \mathbf{x} \in D,\\
& \nabla \cdot \mathbf{u} = 0,\quad \ \   & \mathbf{x} \in D.
\end{aligned}
\right.
\end{equation}
\begin{figure}[H]
	\centering
	\includegraphics[width=0.6\textwidth]{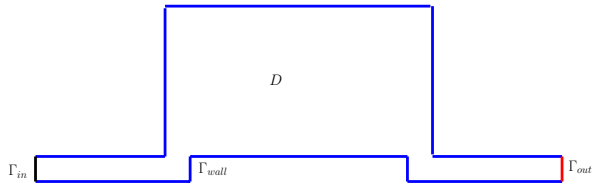} 
	\caption{The domain $D$ for shape optimization based on the phase-field approach with different boundaries: the inflow boundary $\left(\Gamma_{in}\right)$, the wall boundary $\left(\Gamma_{wall}\right)$, and the outflow boundary $\left(\Gamma_{out}\right)$.}
	\label{fig:domain-pf}
\end{figure} 

The boundary $\partial D$ consists of three parts (e.g. Fig.\ref{fig:domain-pf}): $\partial D=\Gamma_{in} \cup \Gamma_{wall} \cup \Gamma_{out}$. The inflow boundary $\Gamma_{in}$ and the outflow boundary are same to the definition of the original model, while the wall boundary satisfies $\Gamma_{wall}= \partial D \backslash (\Gamma_{in} \cup \Gamma_{out})$. 
Same boundary conditions are offered as the original model. An optimized shape is represented implicitly by the phase-field function (e.g. $\phi=0.5$.)

The total amount of fluid that is avilable for the design is given by $\beta V_0$. Consequently, the phase-field function is subject to the constraint
\begin{equation} \label{eq:volumv-constrain}
\int_D \phi d \mathbf{x}= \beta V_0,
\end{equation}
where $V_0$ is the volume of the total designed domain $D$ and $\beta$ is a given positive constant in $[0,1]$. 
 
By introducing a Lagrangian form with respect to the constraint of a given volume, the phase-field model for shape optimization is ultimately given as: 
\begin{equation} \label{eq:min-obj}
\begin{aligned}
& \min_{\phi \in [0,1]}{L(\phi,\mathbf{u},\lambda):= J_{\alpha}(\phi,\mathbf{u})+\eta J_{\epsilon}(\phi) +\lambda \left(\int_{D} \phi d\mathbf{x} - \beta V_0\right),} \\
&  \quad\quad\quad\quad\quad\quad\quad\quad\text{s.t.}  \quad (\ref{eq:state-equation-phi}),
\end{aligned}
\end{equation}
where $\lambda$ is a Lagrange multiplier for the volume constraint.

\subsubsection{The gradient flow approach}
To solve the phase-field model (\ref{eq:min-obj}), we adopt the gradient flow approach. The phase-field equation is derived via $L^2$ gradient flow as,
\begin{equation} \label{eq:phase-field-equation}
\left\{
\begin{aligned}
&\frac{d \phi}{d t} = - \frac{\delta L(\phi,\mathbf{u},\lambda)}{\delta \phi},  \ &  \mathbf{x} \in D, \ t >0,\\
& \phi(\mathbf{x},0) = \phi^0(\mathbf{x}),&   \mathbf{x} \in D,\\
& \frac{\partial \phi}{\partial \mathbf{n}} = 0,& \mathbf{x} \in \partial D,
\end{aligned}
\right.
\end{equation}
where $t$ is the pseudo time and $\phi^0(\mathbf{x}) \left(\in \{\phi | \phi(\mathbf{x})\in [0,1], \mathbf{x}\in D\}\right)$ is the initial guess. The derivation of the objective functional (\ref{eq:min-obj}) can be found in \cite{Takezawa10Shape,Li19A}, which is directly given as
$$  \label{eq:Jphi-derivation}
	\frac{\delta L(\phi, \mathbf{u}, \lambda)}{\delta \phi} = -\epsilon \eta \Delta \phi + \frac{\eta}{\epsilon} \phi (\phi-1)(\phi-\frac{1}{2}) -\alpha_0(1-\phi) \left|\mathbf{u}\right|^2 + \lambda.
$$

Analogously, let $\delta L(\phi,\mathbf{u},\lambda)/\delta \lambda = 0$, and we obtain the volume constrain (\ref{eq:volumv-constrain}). The evolution for the volume constrained condition can be written as
\begin{equation} \label{eq:lambda-equation}
\left\{
\begin{aligned}
&\frac{d \lambda}{dt} = - \frac{\delta L(\phi,\mathbf{u},\lambda)}{\delta \lambda}, \ t >0 \\
& \lambda(0) = \lambda^0,\\
\end{aligned}
\right.
\end{equation}
where 
$$\frac{\delta L(\phi,\mathbf{u},\lambda) }{\delta \lambda} =  \int_{D} \phi dx - \beta V_0$$
 and $\lambda^0$ is the initial guess and is often set to be 0.

It is indeed a complex coupled system that involves the phase-field equation (\ref{eq:phase-field-equation}), the evolution equation (\ref{eq:lambda-equation}) for the volume constrained condition, and the Stokes equation (\ref{eq:state-equation-phi}), which partly stimulates us to construct a decoupled algorithm to implement it in \cref{sec:ENS}. In the iterative procedure, it requires that the phase-field values are in $[0,1]$ for the decoupled schemes in each loop.

\section{Decoupled, energy stable numerical schemes}
\label{sec:ENS}
In this section, an effective decoupled scheme is proposed directly for the coupled system on solving the Stokes equation (\ref{eq:state-equation-scheme-1}), solving the phase-field equation (\ref{eq:semi-implicit-allen-cahn-equation-scheme-2}), using the cut-off postprocessing strategy for the phase-field function (\ref{eq:cut-off-phase-field-scheme-3}), and updating the Lagrangian parameter for the constrained volume condition (\ref{eq:lambda-scheme-4}) in an alternative way. 

Given an initial guess $\phi^0$ and $\lambda^0$, all variables in the system are updated one by one as follows,
$$
\left(\mathbf{u}^1, \phi^{1,*}, \phi^{1}, \lambda^1\right),\left(\mathbf{u}^2, \phi^{2,*}, \phi^{2}, \lambda^2\right), \cdots,\left(\mathbf{u}^n, \phi^{n,*}, \phi^{n}, \lambda^n\right), \cdots.
$$
All Specific details are present in the Algorithm \ref{alg:total-scheme-semi-discrete}.
\begin{algorithm}[H] 
	\caption{Phase-field model for shape optimization}
	\label{alg:total-scheme-semi-discrete}
	Given the parameters $N$, $K$, $\alpha_0$, $\epsilon$, $\eta$,  $\beta$, $\Delta t$, initialize $\phi^{0}$, $\lambda^{0}$, and  $n=0$.
	
	\textbf{while} $n < N$  \textbf{do}
	\begin{enumerate} 
		\item Update the velocity field $\mathbf{u}^{n+1}$ by the following state equation using the known phase-field function in the previous step,
		\begin{equation} \label{eq:state-equation-scheme-1}
			\left\{
			\begin{aligned}
				& - \Delta \mathbf{u}^{n+1} + \nabla p^{n+1} + \alpha(\phi^{n}) \mathbf{u}^{n+1} = 0, \quad   &\mathbf{x} \in D,\\
				& \nabla \cdot \mathbf{u}^{n+1} = 0,\quad \ \  & \mathbf{x} \in D.
			\end{aligned}
			\right.
		\end{equation}	
		\item Set $\phi^{n+1,0}=\phi^n$, $\lambda^{n+1,0}=\lambda^n$ and $k=0$. 
		
		\textbf{while} $k < K$  \textbf{do}
		\begin{enumerate}[(1)]
			\item Update the phase-field variable with the following stabilizing scheme and obtain $\phi^{n+1,k+1,*}$,
			\begin{equation} \label{eq:semi-implicit-allen-cahn-equation-scheme-2}
				\begin{aligned}
					&\frac{\phi^{n+1,k+1,*}}{\Delta t} - \epsilon \eta \Delta \phi^{n+1,k+1,*} + \frac{\alpha_0|\mathbf{u}^{n+1}|^2}{2} \phi^{n+1,k+1,*}  + \tilde{S} \phi^{n+1,k+1,*}  \\
					=& \frac{\phi^{n+1,k}}{\Delta t} - \frac{\eta}{\epsilon} f(\phi^{n+1,k}) - \alpha_0|\mathbf{u}^{n+1}|^2 \left(\frac{\phi^{n+1,k}}{2}-1\right)  -  \lambda^n   + \tilde{S} \phi^{n+1,k},
				\end{aligned}
			\end{equation}
			where $\tilde{S}$ is a stabilized parameter introduced in \cite{Shen10Numerical}.
			\item Restrict the phase-field function in $[0,1]$ by using a cut-off postprocessing strategy and obtain $ \phi^{n+1,k+1}$
			\begin{equation} \label{eq:cut-off-phase-field-scheme-3}
				\phi^{n+1,k+1}(\mathbf{x})=\mathcal{P}(\phi^{n+1,k+1,*}(\mathbf{x})):=\min(\max(\phi^{n+1,k+1,*}(\mathbf{x}),0),1),
			\end{equation}
			where the sign $\mathcal{P}$ denotes a cut-off operator.
			\item Update the Lagrange multiplier $\lambda^{n+1,k+1}$ by 
			\begin{equation} \label{eq:lambda-scheme-4}
				\lambda^{n+1,k+1}= \lambda^{n+1,k}-
				\left\{
				\begin{aligned}
					&
					\beta_0 J_v(\phi^{n+1,k+1}) ,\quad  
					J_v(\phi^{n+1,k+1,*}) \geq J_v(\phi^{n+1}) , \\
					&  \beta_0^* J_v(\phi^{n+1,k+1}),    
					\quad J_v(\phi^{n+1,k+1,*}) < J_v(\phi^{n+1,k+1}) ,
				\end{aligned}
				\right.
			\end{equation} 
			where $\beta_0$ denotes a fixed proper time step, $ \beta_0^*$ is a adjusted variable time step as $ \lambda^{n+1,k} (\int_D \phi^{n+1,k+1} - \phi^{n+1,k+1,*} dx) /( J_v(\phi^{n+1,k+1}))^2$, and $J_v(\phi)$ is defined by $ \int_D \phi d\mathbf{x} - \beta V_0$.
		\end{enumerate}
		\textbf{end while}
	\end{enumerate}
	Set $\phi^{n+1}=\phi^{n+1,K}$, $\lambda^{n+1}=\lambda^{n+1,K}$, and go to Step 1.
	
	\textbf{end while}
\end{algorithm}

\begin{remark}
	It is noted that Algorithm \ref{alg:total-scheme-semi-discrete} updates the phase field variable and the Lagrange multiplier $K$ times in the inner iteration in each loop. This inner iteration can enhance the overall efficiency significantly.
\end{remark}

In the following part, we will mainly illustrate the monotonic-decaying property step by step. We just prove the case when $K=1$ and omit the index $k$ for the inner iteration for simplicity.
$$
\begin{aligned}
 L\left(\phi^{n+1}, \mathbf{u}^{n+1},\lambda^{n+1}\right)  
&\xrightarrow[\textbf{LEMMA \ref{lem:3}}]{\mathbf{\leq}}
L\left(\phi^{n+1,*},\mathbf{u}^{n+1},\lambda^{n}\right)  \\
&\xrightarrow[\textbf{LEMMA \ref{lem:2}}]{\mathbf{\leq}}
L\left(\phi^{n},\mathbf{u}^{n+1},\lambda^{n}\right) \\
&\xrightarrow[\textbf{LEMMA \ref{lem:1}}]{\mathbf{\leq}}
L\left(\phi^{n},\mathbf{u}^{n},\lambda^{n}\right),
\end{aligned}
$$
for $n=0,1,2,\cdots$.

\subsection{Updating the state variable}
We now update the velocity field via solving a Stokes equation based on the fact that the known phase field is given. It clearly illustrates that the objective functional holds a monotonic-decaying property with respect to the velocity field in LEMMA \ref{lem:1}.
\begin{lemma} \label{lem:1}
	The phase field $\phi^{n}$ and the Lagrangian parameter $\lambda^{n}$ are given, by soving the state equation $(\ref{eq:state-equation-scheme-1})$, we have
	\begin{equation} \label{eq:state-decay}
		L(\phi^{n},\mathbf{u}^{n+1},\lambda^{n}) \leq L(\phi^{n},\mathbf{u}^{n},\lambda^{n}), \quad n=1,2,\cdots.
	\end{equation}
\end{lemma}
\begin{proof}
	First, $\mathbf{u}^{n+1}$ satisfies the following Stokes equation, 
	\begin{equation} \label{eq:update-state-proof-1}
	\left\{
	\begin{aligned}
	& - \Delta \mathbf{u}^{n+1} + \nabla p^{n+1} + \alpha(\phi^{n}) \mathbf{u}^{n+1} = 0, \quad   &\mathbf{x} \in D,\\
	& \nabla \cdot \mathbf{u}^{n+1} = 0,\quad \ \  & \mathbf{x} \in D.
	\end{aligned}
	\right.
	\end{equation}	
	Take $\mathbf{u}^{n}-\mathbf{u}^{n+1}$ as the test function in (\ref{eq:update-state-proof-1}), integrate over the domain $D$, and we obtain
	\begin{equation} \label{eq:update-state-proof-2}
	\begin{aligned}
	 (\nabla \mathbf{u}^{n+1},\nabla \mathbf{u}^{n}-\nabla\mathbf{u}^{n+1}) + (\alpha(\phi^{n}) \mathbf{u}^{n+1},\mathbf{u}^{n}-\mathbf{u}^{n+1}) = 0, \quad   &\mathbf{x} \in D.
	\end{aligned}
	\end{equation}
	On the other hand, combining (\ref{eq:update-state-proof-2}) and the known identity $\frac{1}{2}(|a|^2-|b|^2) = (a-b,b)+\frac{1}{2}|a-b|^2$, we have the estimation
	\begin{equation} \label{eq:update-state-proof-3}
	\begin{aligned}
	&L(\phi^{n},\mathbf{u}^{n},\lambda^{n}) - L(\phi^{n},\mathbf{u}^{n+1},\lambda^{n}) \\
	=&\frac{1}{2}(\|\nabla \mathbf{u}^n\|^2 -\|\nabla \mathbf{u}^{n+1}\|^2) + \frac{1}{2}   \left(\left\|\sqrt{\alpha(\phi^n)} \mathbf{u}^{n}\right\|^2-\left\| \sqrt{\alpha(\phi^n)}\mathbf{u}^{n+1}\right\|^2\right) \\
	=& (\nabla \mathbf{u}^n - \nabla \mathbf{u}^{n+1},\nabla \mathbf{u}^{n+1}) + \frac{1}{2}\|\nabla \mathbf{u}^n - \nabla \mathbf{u}^{n+1}\|^2  \\
	& +  \left( \alpha(\phi^n) \mathbf{u}^{n+1}, \mathbf{u}^{n}-\mathbf{u}^{n+1} \right) + \frac{1}{2}\alpha(\phi^n) \left\|\mathbf{u}^{n}-\mathbf{u}^{n+1}\right\|^2 \\
	= & \frac{1}{2}\|\nabla \mathbf{u}^n - \nabla \mathbf{u}^{n+1}\|^2 + \frac{1}{2}\alpha(\phi^n) \left\|\mathbf{u}^{n}-\mathbf{u}^{n+1}\right\|^2 \\
	\geq &0,
	\end{aligned}
	\end{equation}
	where $\|\cdot\|$ donates the $L^2$-norm in the domain $D$.
	 Therefore, 
	 it implies the result (\ref{eq:state-decay}).
\end{proof}

\subsection{Updating the phase-field variable}
For the sake of simplicity, we just consider the one-step iteration and prove the case when $K=1$ and omit the index $k$ for the inner iteration.
LEMMA \ref{lem:2} shows the monotonic-decaying property with respect to the phase-field variable when solving the phase-field equation.
\begin{lemma} \label{lem:2}
Under the condition $\tilde{S}\geq \frac{\eta}{4\epsilon}$, the scheme $(\ref{eq:semi-implicit-allen-cahn-equation-scheme-2})$ is unconditionally energy stable, moreover,
\begin{equation}
L(\phi^{n+1,*},\mathbf{u}^{n+1},\lambda^{n}) \leq L(\phi^{n},\mathbf{u}^{n+1},\lambda^{n}), \quad n=0,1,\cdots,N.
\end{equation}
\end{lemma}
\begin{proof}
Take $\phi^{n+1,*} - \phi^n$ as the test function, integrate over $D$ in the scheme $(\ref{eq:semi-implicit-allen-cahn-equation-scheme-2})$, and we obtain
\begin{equation} \label{eq:phase-field-proof-pf}
\begin{aligned}
& \frac{1}{\Delta t}(\phi^{n+1,*},\phi^{n+1,*}-\phi^{n}) -  \epsilon \eta (\Delta \phi^{n+1,*}, \phi^{n+1,*}-\phi^{n}) +  (\lambda^n, \phi^{n+1,*}-\phi^{n}) \\
& +  \frac{ \alpha_0 }{2}( |\mathbf{u^{n+1}}|^2 \phi^{n+1,*},\phi^{n+1,*}-\phi^{n})  
+  \tilde{S}(\phi^{n+1,*},\phi^{n+1,*}-\phi^{n}) \\
=& \frac{1}{\Delta t}(\phi^{n},\phi^{n+1,*}-\phi^{n}) -   \frac{\eta} {\epsilon}(f(\phi^n),\phi^{n+1,*}-\phi^{n}) \\
& -  \frac{ \alpha_0 }{2}\left( |\mathbf{u}^{n+1}|^2 (\phi^{n}-2),\phi^{n+1,*}-\phi^{n}\right)  
+  \tilde{S}(\phi^{n},\phi^{n+1,*}-\phi^{n}).
\end{aligned}
\end{equation}
On the other hand, under the condition (\ref{eq:phase-field-proof-pf}), we merely need to prove an equivalent inequality that $L(\phi^{n+1,*},\mathbf{u}^{n+1},\lambda^{n}) - L(\phi^{n},\mathbf{u}^{n+1},\lambda^{n}) \leq 0$.

We first consider the penalization energy, and have
\begin{equation} \label{eq:j2-constrian}
\begin{aligned}
& \int_D\frac{1}{2} \alpha_0 (1-\phi^{n+1,*})^2 |\mathbf{u}^{n+1}|^2 d \mathbf{x} -\int_D \frac{1}{2}\alpha_0 (1-\phi^{n})^2 |\mathbf{u}^{n+1}|^2 d \mathbf{x} \\
=&  \frac{\alpha_0}{2} \left( |\mathbf{u}^{n+1}|^2 \phi^{n+1,*},\phi^{n+1,*}-\phi^{n} \right) +  \frac{\alpha_0}{2} \left( |\mathbf{u}^{n+1}|^2(\phi^{n}-2),\phi^{n+1,*}-\phi^{n} \right)
\end{aligned}
\end{equation}
Next using the identity $\frac{1}{2}(|a|^2-|b|^2)=(a-b,a)-\frac{1}{2}|a-b|^2$ and partial integral formula, we obtain the following estimation
\begin{equation} \label{eq:laplace-constrian-2}
\begin{aligned}
&\int_D \frac{1}{2}\left|\nabla \phi^{n+1,*} \right|^2d \mathbf{x} - \int_D \frac{1}{2} \left|\nabla \phi^{n} \right|^2d \mathbf{x} = \frac{1}{2} \|\nabla \phi^{n+1,*} \|^2_{L^2(D)} -  \frac{1}{2} \|\nabla \phi^{n} \|^2_{L^2(D)} \\
=& (-\Delta \phi^{n+1,*}, \phi^{n+1,*}-\phi^{n}) - \frac{1}{2} \|\nabla \phi^{n+1,*} - \nabla \phi^n\|^2 
\end{aligned}
\end{equation}

For the double-well function, a truncated technique is taken directly in the analysis of the Allen-Cahn-type equation,
$$
F(\phi)=
\left\{
\begin{aligned}
& \frac{1}{4}(\phi-1)^2, \quad \phi > 1,  \\
& \frac{1}{4}\phi^2(\phi-1)^2, \quad  0 \leq \phi \leq 1, \\
& \frac{1}{4}\phi^2, \quad \phi<0,
\end{aligned}
\right.
$$
where $f^{\prime}(\phi)$ satisfies the following condition
$
\max_{\mathbf{x}\in \mathcal{R}^d} |f^{\prime}(\phi)| \leq \frac{1}{2}.
$
Combining the Taylor expansion of double-well function and the truncated technique, integrating over $D$, we have the following inequality  
\begin{equation}
\label{eq:volume-Taylor-expansion}
\left(F(\phi^{n+1,*}) - F(\phi^n),1\right) \leq (f(\phi^n), \phi^{n+1,*}-\phi^n) + \frac{1}{4}\|\phi^{n+1,*}- \phi^n\|^2.
\end{equation}

Define that $J_v(\phi):= \int_D \phi d\mathbf{x} - \beta V_0$. For the volume energy, we have the following estimation,
\begin{equation} \label{eq:jv-constrian}
\begin{aligned}
 J_v(\phi^{n+1,*})-J_v(\phi^n) 
= \int_D \phi^{n+1,*}-\phi^{n} \ dx 
=(1, \phi^{n+1,*}-\phi^{n}).
\end{aligned}
\end{equation}

Combining (\ref{eq:phase-field-proof-pf}), (\ref{eq:j2-constrian}), (\ref{eq:laplace-constrian-2}), (\ref{eq:volume-Taylor-expansion}) and (\ref{eq:jv-constrian}), we obtain
\begin{equation}
\begin{aligned}
 & L\left(\phi^{n+1,*},\mathbf{u}^{n+1},\lambda^{n}\right)-L\left(\phi^{n},\mathbf{u}^{n+1},\lambda^{n}\right) \\
= & \frac{\alpha_0 }{2}  \left(|\mathbf{u}^{n+1}|^2 \phi^{n+1,*},\phi^{n+1,*}-\phi^{n}\right) +  \frac{\alpha_0}{2}  \left(|\mathbf{u}^{n+1}|^2 (\phi^{n}-2),\phi^{n+1,*}-\phi^{n}\right) \\
& + \frac{\eta \epsilon}{2} \left(\|\nabla \phi^{n+1,*}\|^2 -\|\nabla \phi^{n}\|^2\right) + \frac{\eta}{\epsilon}  \left( F(\phi^{n+1,*}) - F(\phi^n),1  \right)  + \lambda^{n}(1, \phi^{n+1,*}-\phi^n) \\
\leq & \left(-\frac{1}{\Delta t} -\tilde{S} +  \frac{\eta }{4 \epsilon} \right)\|\phi^{n+1,*}-\phi^n\|^2 - \frac{\eta \epsilon}{2}\|\nabla \phi^{n+1,*}-\nabla \phi^n\|^2,  
\end{aligned}
\end{equation}
which implies the proof if $\tilde{S}\geq \frac{\eta}{4\epsilon}$.
\end{proof}

\subsection{Revising the phase-field function and updating the Lagrangian parameter}
To make sure that $\phi^{n+1}\in [0,1]$ $(n=0,1,2,\cdots)$ after the calculation of phase-field function in a loop, a natural idea is to apply a cut-off postprocessing strategy for the phase-field function directly, such as (\ref{eq:cut-off-phase-field-scheme-3}). However, it comes to the volume term, a cut-off postprocessing may lead to an increase in the volume constraint.  Variable time step is adopted to erase the increment in this scheme. Then
the energy holds the monotonic-decaying property in LEMMA \ref{lem:3}.

\begin{lemma} \label{lem:3}
After updating the phase-field variable $\phi^{n+1}$ by a cut-off postprocessing strategy $(\ref{eq:cut-off-phase-field-scheme-3})$ and assuming that $ \int_D \phi^{n+1} d\mathbf{x} - \beta V_0\neq 0$, we have the discrete energy law under a proper variable time step $(\ref{eq:lambda-scheme-4})$ for the update of the Lagrangian parameter,
\begin{equation}
L\left(\phi^{n+1},\mathbf{u}^{n+1},\lambda^{n+1}\right) \leq L\left(\phi^{n+1,*},\mathbf{u}^{n+1},\lambda^{n}\right).
\end{equation}
\end{lemma}
\begin{proof}
By a cut-off postprocessing strategy (\ref{eq:cut-off-phase-field-scheme-3}) for the phase-field function, we demonstrate that the objective functional is nonincreasing. 

Let $D=D_1\cup D_2$ and $D_1 \cap D_2 = \emptyset$, where $D_1$ donates the domain that the phase-field values over $[0,1]$,  then we obtain 
$$
\nabla \phi^{n+1} = \mathbf{0} \quad \mathbf{x} \in D_1, \quad  \nabla \phi^{n+1} = \nabla \phi^{n+1,*} \quad \mathbf{x} \in D_2.
$$
Therefore, we have the following estimation,
$$
\begin{aligned}
\left\|\nabla \phi^{n+1} \right\|_{L^2(D)}^2 &=  \left\|\nabla \phi^{n+1} \right\|_{L^2(D_2)}^2  \\
&\leq \left\|\nabla \phi^{n+1,*} \right\|_{L^2(D_1)}^2 + \left\|\nabla \phi^{n+1} \right\|_{L^2(D_2)}^2  \\
&= \left\|\nabla \phi^{n+1,*} \right\|_{L^2(D_1)}^2 + \left\|\nabla \phi^{n+1,*} \right\|_{L^2(D_2)}^2  \\
& = \left\|\nabla \phi^{n+1,*}\right\|_{L^2(D)}^2,
\end{aligned}
$$
which also illustrate that the it doesn't make a contribution for the gradient operator in the domain $D_1$ due to a cut-off postprocessing. Then we have
\begin{equation} \label{eq:proof-12}
\int_D \frac{\epsilon}{2} \left|\nabla \phi^{n+1}\right|^2  d \mathbf{x} \leq \int_D \frac{\epsilon}{2} \left|\nabla \phi^{n+1,*}\right|^2  d \mathbf{x}.
\end{equation}

In the objective functional of the phase-field model (\ref{eq:min-obj}), for the penalized term and the double-well function, we directly yields
\begin{equation} \label{eq:proof-11}
\begin{aligned}
 \int_D \frac{1}{2}(1-\phi^{n+1})^2 |\mathbf{u}^{n+1}|^2 dx 
\leq  \int_D \frac{1}{2}(1-\phi^{n+1,*})^2 |\mathbf{u}^{n+1}|^2 dx,
\end{aligned}
\end{equation}
and
\begin{equation}
\int_D (\phi^{n+1})^2(\phi^{n+1}-1)^2dx \leq \int_D(\phi^{n+1,*})^2(\phi^{n+1,*}-1)^2dx.
\end{equation}

For the volume term, we first define that $J_v(\phi):= \int_D \phi d\mathbf{x} - \beta V_0$ and update $\lambda^{n+1}$ by $\lambda^{n+1}= \lambda^{n}-\beta_0 J_v(\phi^{n+1})$ if $J_v(\phi^{n+1,*}) \geq J_v(\phi^{n+1})$.
When $J_v(\phi^{n+1,*}) < J_v(\phi^{n+1})$, the increment requires erasing via adjusting the parameter $\beta_0^*$ such that
\begin{equation}
\begin{aligned}
& \lambda^{n+1} J_v(\phi^{n+1}) \leq  \lambda^{n} J_v(\phi^{n+1,*}) \\
\Longleftrightarrow & \left(\lambda^{n}-\beta_0^* J_v(\phi^{n+1}) \right) J_v(\phi^{n+1}) \leq  \lambda^{n} J_v(\phi^{n+1,*}) \\
\Longleftrightarrow & \lambda^{n} \left(  \int_D \phi^{n+1} - \phi^{n+1,*} dx \right) \leq \beta_0^* \left(J_v(\phi^{n+1}) \right)^2 \\
\Longleftrightarrow & \beta_0^*  \geq  \lambda^{n} \frac{ \int_D \phi^{n+1} - \phi^{n+1,*} dx }{\left(J_v(\phi^{n+1})\right)^2},    \\
\end{aligned}
\end{equation}
 which makes sure that $J_v(\phi^{n+1,*}) \geq J_v(\phi^{n+1})$.
 
 Therefore, we complete the proof.
\end{proof}

\begin{theorem}
When using the numerical Algorithm \ref{alg:total-scheme-semi-discrete} for the objective problem $(\ref{eq:min-obj})$, the total energy holds the monotonic-decaying property,
\begin{equation}
L\left(\phi^{n+1},\mathbf{u}^{n+1},\lambda^{n+1}\right) \leq L\left(\phi^{n},\mathbf{u}^{n},\lambda^{n}\right).
\end{equation}
\end{theorem}
\begin{proof}
The proof follows by the forthcoming LEMMA \ref{lem:1}-LEMMA \ref{lem:3}.
\end{proof}

\section{Fully discretization finite element method}
\label{sec:FED}
We first introduce the weak formulations of Stokes equation (\ref{eq:state-equation-scheme-1}) and phase-field equation (\ref{eq:semi-implicit-allen-cahn-equation-scheme-2}) in our algorithm, respectively. They are described as follows,
to find $\mathbf{u} \in H^1(D)$ and $p\in L^2(D)$ under the Dirichlet condition $\mathbf{u}=\mathbf{g}$ on the $\Gamma_{in}$ such that
\begin{equation}
	\label{eq:state-equation-weak-form}
		\left\{
		\begin{aligned}
			& (\nabla \mathbf{u}^{n+1},\nabla \mathbf{w}) + (\nabla p^{n+1},\mathbf{w}) + (\alpha(\phi^n) \mathbf{u}^{n+1},\mathbf{w}) = 0,    & \mathbf{w} \in H^1_0(D),\\
			& (\nabla \cdot \mathbf{u}^{n+1},q) = 0,  & q \in L^2(D),
		\end{aligned}
		\right.
\end{equation}
and to find $\phi\in H^1(D)$ such that
\begin{equation}
\label{eq:allen-cahn-equation-weak-form}
\begin{aligned}
	&\frac{1}{\Delta t}(\phi^{n+1,*},\psi) + \epsilon \eta (\nabla \phi^{n+1,*},\nabla \psi) + \left( \left(\frac{\alpha_0|\mathbf{u}^{n+1}|^2}{2} + \tilde{S}\right)\phi^{n+1,*},\psi \right)   \\
	=& \frac{1}{\Delta t} (\phi^{n},\psi) + \left( - \frac{\eta}{\epsilon} f(\phi^{n}) + \alpha_0|\mathbf{u}^{n+1}|^2 - \lambda^n,\psi \right)  + \left(  \left(\tilde{S} - \frac{1}{2}\alpha_0|\mathbf{u}^{n+1}|^2\right) \phi^{n}, \psi \right), 
\end{aligned}
\end{equation}
where $\psi \in H^1_0(D)$.

Assume that $\mathcal{T}_h$ is a given mesh, for instance, triangular mesh in two dimensions and tetrahedral mesh in three dimensions.
We donate $P_1(\mathcal{T}_h)$ as the piecewise linear element space and $P_2(\mathcal{T}_h)$ as the quadratic element space defined on the mesh $\mathcal{T}_h$, respectively. Before describing the fully-discrete algorithm, we first give some definitions on the finite-dimensional spaces,
$$
\begin{aligned}
	&\mathcal{P}_h(\mathcal{T}_h):= H^1(D) \cap P_1(\mathcal{T}_h), \quad \quad \mathcal{P}_{h,0}(\mathcal{T}_h):= H^1_0(D) \cap P_1(\mathcal{T}_h), \\
	&\mathcal{Q}_h(\mathcal{T}_h):= H^1(D) \cap P_2(\mathcal{T}_h), \quad \quad \mathcal{Q}_{h,0}(\mathcal{T}_h):= H^1_0(D) \cap P_2(\mathcal{T}_h), \\
	&\mathcal{M}_{h}(\mathcal{T}_h):= L^2(D) \cap P_1(\mathcal{T}_h).
\end{aligned}
$$

\subsection{The fully-discrete version}
In this paper, we use the Taylor-Hood (P2-P1) element to solve the Stokes equation and employ the piecewise finite element to approximate the solution of phase-field equation on the same mesh $\mathcal{T}_h$.
In our designed decoupled algorithm (see, for instance, Algorithm \ref{alg:total-scheme-fully-discrete}), the Stokes equation (\ref{eq:fully-state-equation-scheme-1}) and the Allen-Cahn-type equation (\ref{eq:fully-semi-implicit-allen-cahn-equation-scheme-2}) are solved in an alternative way. In each loop, we compute the velocity field by solving the Stokes equation with the phase-field variable given in the previous iteration and update the phase-field variable by solving the phase-field equation with the velocity field given in the previous solution. However, the solutions are mismatched between the velocity field and the phase field due to employing different polynomials as the basis functions on the same mesh. Projection and interpolation (e.g., \cite{Wu18A}) are simply used to transport datas, which ensures the implementation of our designed decoupled scheme. (For instance, the finite element values of the two variables in 2D are described in Fig.\ref{fig:communication-p1-p2-element}.)
\begin{figure}[H]
	\centering
	\includegraphics[width=0.35\textwidth]{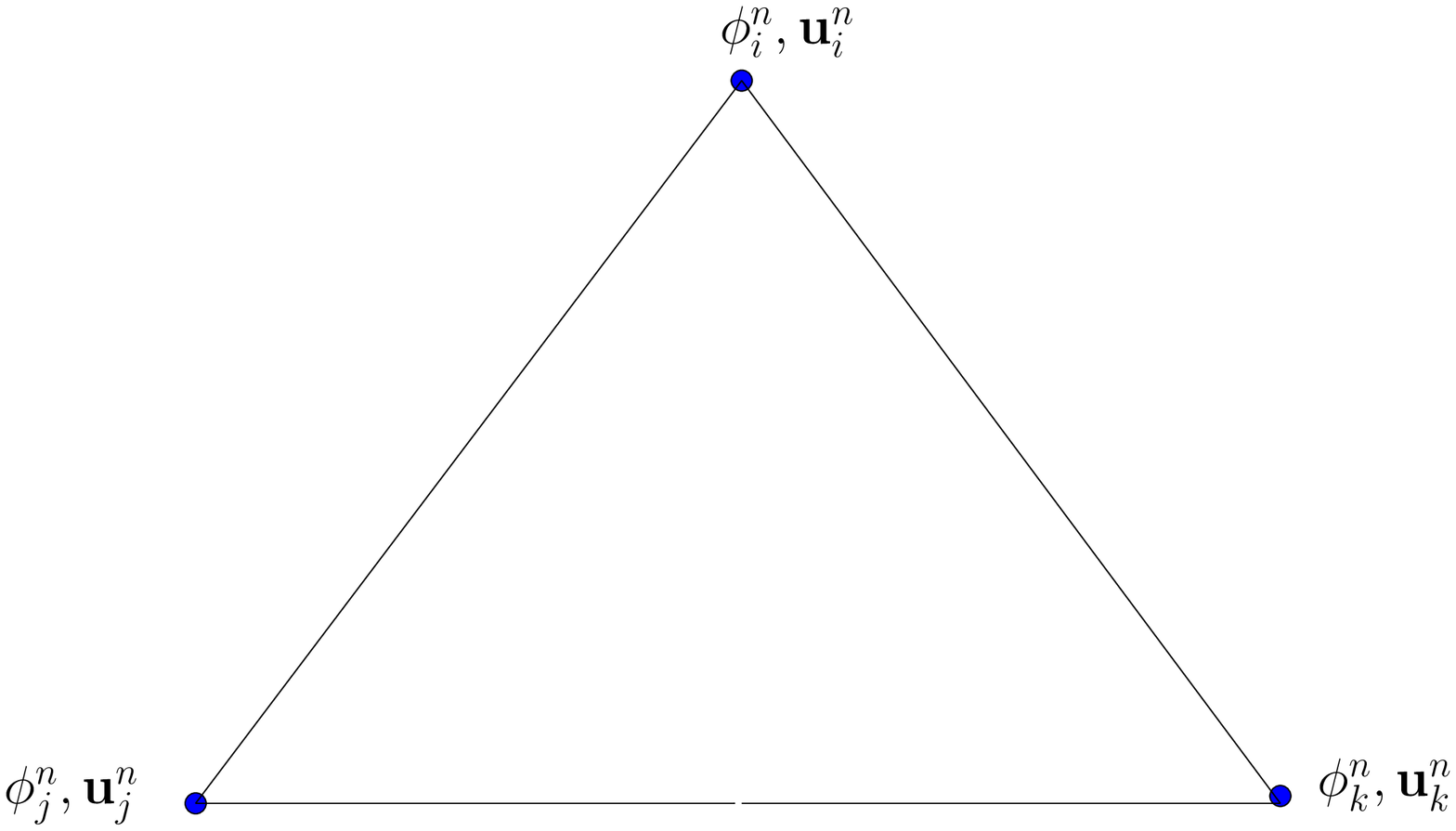} \hspace{1cm}
	\includegraphics[width=0.35\textwidth]{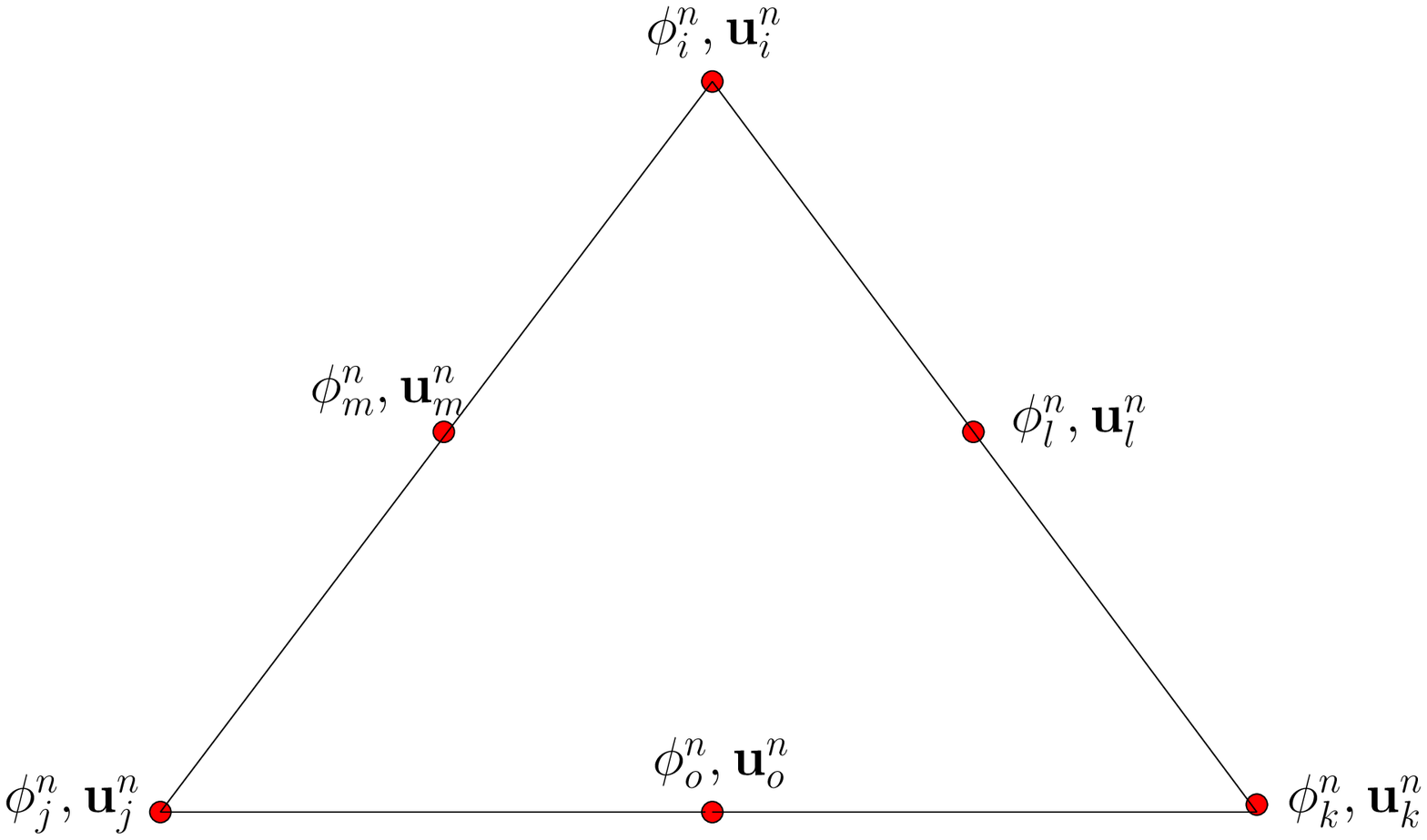} 
	\caption{Projection and interpolation between P1 element and P2 element.}
	\label{fig:communication-p1-p2-element}
\end{figure}

\begin{algorithm}[!h] 
	\caption{Phase-field model for shape optimization in the fully-discrete version}
	\label{alg:total-scheme-fully-discrete}
	Given the parameters $N$, $K$, $\alpha_0$, $\epsilon$, $\eta$, $\beta$, $\Delta t$, initialize $\phi^{0}_h$, $\lambda^{0}$, and set $n=0$.
	
	\textbf{while} $n < N$  \textbf{do}
	\begin{enumerate} 
		\item Update the velocity field $\mathbf{u}^{n+1}_h$ $(\in \mathcal{Q}_{h}(\mathcal{T}_h))$ by the following state equation using the Taylor-Hood element,
		\begin{equation} \label{eq:fully-state-equation-scheme-1}
			\left\{
			\begin{aligned}
				& (\nabla \mathbf{u}^{n+1}_h,\nabla \mathbf{w}_h) + (\nabla p^{n+1}_h,\mathbf{w}_h) + (\alpha(\phi^n_h) \mathbf{u}^{n+1}_h,\mathbf{w}_h) = 0,    & \mathbf{w}_h \in \mathcal{Q}_{h,0}(\mathcal{T}_h),\\
				& (\nabla \cdot \mathbf{u}^{n+1}_h,q_h) = 0,  & q_h \in \mathcal{M}_{h}(\mathcal{T}_h).
			\end{aligned}
			\right.
		\end{equation}	
		\item Set $\phi^{n+1,0}_h=\phi^n_h$, $\lambda^{n+1,0}=\lambda^n$ and $k=0$. 
		
		\textbf{while} $k < K$  \textbf{do}
		\begin{enumerate}[(1)]
			\item Update the phase-field variable with the piecewise linear element and obtain $\phi^{n+1,k+1,*}_h$ $(\in \mathcal{P}_{h}(\mathcal{T}_h))$,
			\begin{equation} \label{eq:fully-semi-implicit-allen-cahn-equation-scheme-2}
				\begin{aligned}
					&\frac{1}{\Delta t}(\phi^{n+1,k+1,*}_h,\psi_h) + \epsilon \eta (\nabla \phi^{n+1,k+1,*}_h,\nabla \psi_h) \\
					&+ \left( \left(\frac{\alpha_0|\mathbf{u}^{n+1}_h|^2}{2} + \tilde{S}\right)\phi^{n+1,k+1,*}_h,\psi_h \right)   \\
					=& \frac{1}{\Delta t} (\phi^{n+1,k}_h,\psi_h) + \left( - \frac{\eta}{\epsilon} f(\phi^{n+1,k}_h) + \alpha_0|\mathbf{u}_h^{n+1}|^2 - \lambda^n,\psi_h \right) \\
					& \left(  \left(\tilde{S} - \frac{\alpha_0|\mathbf{u}_h^{n+1}|^2}{2} \right) \phi^{n+1,k}_h, \psi_h \right), \psi_h \in \mathcal{P}_{h,0}(\mathcal{T}_h),
				\end{aligned}
			\end{equation}
			where $\tilde{S}$ is a stabilized parameter introduced in \cite{Shen10Numerical}.
			\item Restrict the phase-field function in $[0,1]$ by using a cut-off postprocessing strategy and obtain $ \phi^{n+1,k+1}_h$ 
			\begin{equation} \label{eq:fully-trunction-fully-discrization}
				\phi^{n+1,k+1}_{h}(\mathbf{x})=\mathcal{P}(\phi^{n+1,k+1,*}_{h}(\mathbf{x})):=\min(\max(\phi^{n+1,k+1,*}_h(\mathbf{x}),0),1),
			\end{equation}
			where the sign $\mathcal{P}$ denotes a cut-off operator.
			\item Update the Lagrange multiplier $\lambda^{n+1,k+1}$ by 
			\begin{equation} \label{eq:fully-lambda-scheme-4}
				\lambda^{n+1,k+1}= \lambda^{n+1,k}-
				\left\{
				\begin{aligned}
					&
					\beta_0 J_v(\phi^{n+1,k+1}_h),   
					\ J_v(\phi^{n+1,k+1,*}_h) \geq J_v(\phi^{n+1}_h) , \\
					&  \beta_0^* J_v(\phi^{n+1,k+1}_h),    
					\ J_v(\phi^{n+1,k+1,*}_h) < J_v(\phi^{n+1,k+1}_h) ,
				\end{aligned}
				\right.
			\end{equation} 
			where $\beta_0$ denotes a fixed proper time step, $ \beta_0^*$ is a adjusted variable time step as $ \lambda^{n+1,k} (\int_D \phi^{n+1,k+1}_h - \phi^{n+1,k+1,*}_h dx) /(\int_D \phi^{n+1,k+1}_h dx-\beta V_0)^2$, and $J_v(\phi)$ is defined by $ \int_D \phi d\mathbf{x} - \beta V_0$.
		\end{enumerate}
		\textbf{end while}
	\end{enumerate}
	Set $\phi_h^{n+1}=\phi_h^{n+1,K}$, $\lambda^{n+1}=\lambda^{n+1,K}$, and go to Step 1.
	
	\textbf{end while}
\end{algorithm}

Given an initial guess $\phi^0_h$ and $\lambda^0$, all variables in the full-discrete system are updated one by one as follows,
$$
\left(\mathbf{u}^1_h, \phi^{1,*}_h, \phi^{1}_h, \lambda^1\right),\left(\mathbf{u}^2_h, \phi^{2,*}_h, \phi^{2}_h, \lambda^2\right), \cdots,\left(\mathbf{u}^n_h, \phi^{n,*}_h, \phi^{n}_h, \lambda^n\right), \cdots.
$$
All Specific details are present in the Algorithm \ref{alg:total-scheme-fully-discrete}. 

In the following part, we mainly illustrate the monotonic-decaying property similar to the semi-discrete version that is shown in \cref{sec:ENS}.
$$
\begin{aligned}
	L\left(\phi^{n+1}_h, \mathbf{u}^{n+1}_h,\lambda^{n+1}\right) 
	\leq
	L\left(\phi^{n+1,*}_h,\mathbf{u}^{n+1}_h,\lambda^{n}\right) 
	\leq
	L\left(\phi^{n}_h,\mathbf{u}^{n+1}_h,\lambda^{n}\right) 
	\leq
	L\left(\phi^{n}_h,\mathbf{u}^{n}_h,\lambda^{n}\right),
\end{aligned}
$$
for $n=0,1,2,\cdots$.

Next we just need to deduce the inequality $\left\|\nabla \phi^{n+1}_h \right\|^2 \leq \left\|\nabla \phi^{n+1,*}_h\right\|^2$ in the fully-discrete sense. Other steps of the proof of the fully-discrete decoupled monotonic-decaying scheme are same to the ones in \cref{sec:ENS} and we omit the corresponding contents.

\begin{lemma} \label{lem:4}
	Assume that the phase-field equation $(\ref{eq:fully-semi-implicit-allen-cahn-equation-scheme-2})$ is solved with piecewise linear finite element method. After applying a cut-off postprocessing $(\ref{eq:fully-trunction-fully-discrization})$ to the phase-field variable, we have the following conclusion,
	\begin{equation}
		\left\|\nabla \phi_h^{n+1} \right\|^2 \leq \left\| \nabla \phi_h^{n+1,*} \right\|^2.
	\end{equation}
\end{lemma}

\begin{proof}
	Let us consider only the 2D case (the 3D one is similar). Without loss of generality, we can assume that $T_i$ is an arbitrary triangular element  of mesh $\mathcal{T}_h$ and $T_{*}$ is a reference element with three vertices $(0,0)$, $(0,1)$ and $(1,0)$.  
	It is clear to find that the solution space of phase-field equation is generalized by the linear span of linear functions since we use the typical $P1$ finite element, which implies that the coordinates $(0,0,\phi^{n+1,*}_{*hi1})$, $(0,1,\phi^{n+1,*}_{*hi2})$ and $(1,0,\phi^{n+1,*}_{*hi3})$ are on a plane. Therefore, we have the equation of plane
\begin{equation} \label{eq:plane}
\phi_{*h}^{n+1,*} - \phi_{*hi1} = (\phi^{n+1,*}_{*hi2}-\phi^{n+1,*}_{*hi1})x + (\phi^{n+1,*}_{*hi3}-\phi^{n+1,*}_{*hi1})y, 
\end{equation}
where $\phi^{n+1,*}_{*hi1}$, $\phi^{n+1,*}_{*hi2}$, and $\phi^{n+1,*}_{*hi3}$ are all the nodal values of phase-field function on the reference element $T_{*}$ that is mapped from the $i$th triangular element $T_i$. $\phi_{*h}^{n+1,*}$ denotes the unknown phase-field value on the reference element. 

So using (\ref{eq:plane}), the gradient of phase-field function can be derived as
\begin{equation}\label{eq:gradient-plane}
|\nabla \phi_{*h}^{n+1,*}|^2 = (\phi^{n+1,*}_{*hi2}-\phi^{n+1,*}_{*hi1})^2 + (\phi^{n+1,*}_{*hi3}-\phi^{n+1,*}_{*hi1})^2.
\end{equation}
After a cut-off postprocessing for the phase-field function, the height differences between the arbitrary two nodal values decrease. It tells us that the following inequality satisfies,
\begin{equation}\label{eq:cut-off-relation-gradient-plane}
(\phi^{n+1}_{*hi2}-\phi^{n+1}_{*hi1})^2 \leq   (\phi^{n+1,*}_{*hi2}-\phi^{n+1,*}_{*hi1})^2, \quad (\phi^{n+1}_{*hi3}-\phi^{n+1}_{*hi1})^2 \leq  (\phi^{n+1,*}_{*hi3}-\phi^{n+1,*}_{*hi1})^2.
\end{equation}
Combining (\ref{eq:gradient-plane}) and (\ref{eq:cut-off-relation-gradient-plane}), we have 
\begin{equation} \label{eq:gradient-plane-result}
	|\nabla \phi_{*h}^{n+1}|^2 \leq |\nabla \phi_{*h}^{n+1,*}|^2
\end{equation}

Next we consider an arbitrary triangular element $T_i$, and map linearly it to the reference element that we have discussed. Indeed, the original plane is mapped into a new plane and the phase-field values $\phi^{n+1,*}_{hij}$ $(j=1,2,3)$ have changed into $\phi^{n+1,*}_{*hij}$ $(j=1,2,3)$. The inequality (\ref{eq:cut-off-relation-gradient-plane}) satisfies since the height differences between the arbitrary two nodal values of the mapped reference element still decrease. And then we can obtain the same result (\ref{eq:gradient-plane-result}).

Using (\ref{eq:gradient-plane-result}), taking $\phi_{*h}^{n+1,*}$ as the phase-field values of the three nodal points on the reference element, we obtain
\begin{equation} \label{eq:proof-grdient-cut-off}
\begin{aligned}
\left\|\nabla \phi_h^{n+1} \right\|^2 &= \sum_i \int_{T_i} \sum_{j=1}^3 \left|\nabla \phi_{hij}^{n+1}\right|^2  d\mathbf{x}=\sum_i 2 S_i \int_{T_{*}} \sum_{j=1}^3\left|\nabla \phi_{*hij}^{n+1}\right|^2 d\mathbf{x} \\
& \leq \sum_i 2S_i \int_{T_{*}} \sum_{j=1}^3\left|\nabla \phi_{*hij}^{n+1,*}\right|^2 d\mathbf{x} = \sum_i \int_{T_i} \sum_{j=1}^3\left|\nabla \phi_{hij}^{n+1,*}\right|^2  d\mathbf{x} \\
&= \left\| \nabla \phi_h^{n+1,*} \right\|^2,
\end{aligned}
\end{equation}
where $S_i$ denotes the area of the triangular element $T_i$. In fact, the determinant of Jacobian matrix is simply the ratio of the area of the mapped element $T_i$ to that of the reference element $T_{*}$, which is equal to $2S_i$.

Therefore, we complete the proof.
\end{proof}

\begin{theorem}
	Using the numerical Algorithm \ref{alg:total-scheme-fully-discrete} for the objective problem $(\ref{eq:min-obj})$, the total energy holds the monotonic-decaying property,
	\begin{equation}
		L\left(\phi^{n+1}_h,\mathbf{u}^{n+1}_h,\lambda^{n+1}\right) \leq L\left(\phi^{n}_h,\mathbf{u}^{n}_h,\lambda^{n}\right).
	\end{equation}
\end{theorem}
\begin{proof}
	Combining the derived method of the semi-discrete version and LEMMA \ref{lem:4}, we can directly obtain the conclusion.
\end{proof}

\begin{remark}
If we use high-order finite element methods to approximate the phase-field equation, it is very regrettable that the monotonic-decaying property can't preserve. However, adaptive mesh technique provides a road to improve the computational accuracy with piecewise finite element method.
\end{remark}

\section{Numerical experiments}
\label{sec:Numerical-Simu}
In this section , several cases are presented to validate the effectiveness of our schemes. The basic parameters are given as follows : $\epsilon=0.01$ , $\eta=0.01$ , $\delta t=1$ , $\beta_0=1$ , $K=10$ (inner iteration ).
\subsection{Diffuser}
The first example is the design of a diffuser \cite{Borrvall03Topology}. 
The computational domain is set to be a square $D=[0,1]^2$ and the flow velocity is given as $g = y(1-y)/ 0.25$ at the inlet $\Gamma_i$ on the left. In our calculations, the prescribed volume fraction is set to be $\beta = 0.5$. The initial phase-field function is given as follows,
$$
\phi^0(\mathbf{x})=
\left\{
\begin{aligned}
& 1, \ \mathbf{x} \in \Omega_1,\\
& 0, \ \mathbf{x} \in D \backslash \Omega_1,
\end{aligned}
\right.
$$
where the initial shape is described as $\phi^0(\mathbf{x})=1$ in the domain $\Omega_1$, and $\Omega_1:=\left\{(x,y)|x\leq0.25\right\} \cup \left\{(x,y)| \frac{1}{3} \leq y\leq \frac{2}{3}\right\} \cap D$.
\begin{figure}[h!]
	\centering
	\includegraphics[width=0.3\textwidth]{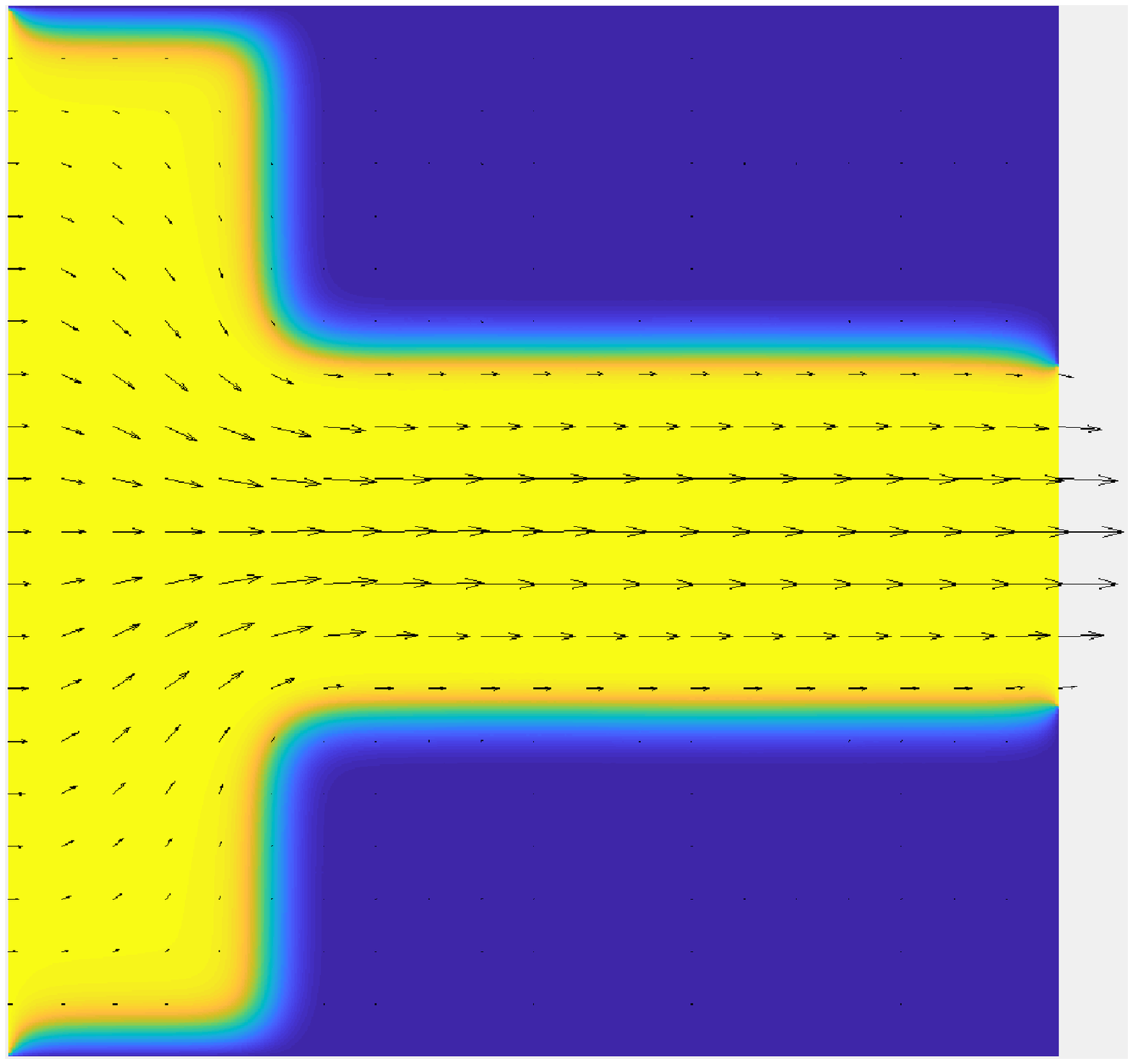} 
	\includegraphics[width=0.3\textwidth]{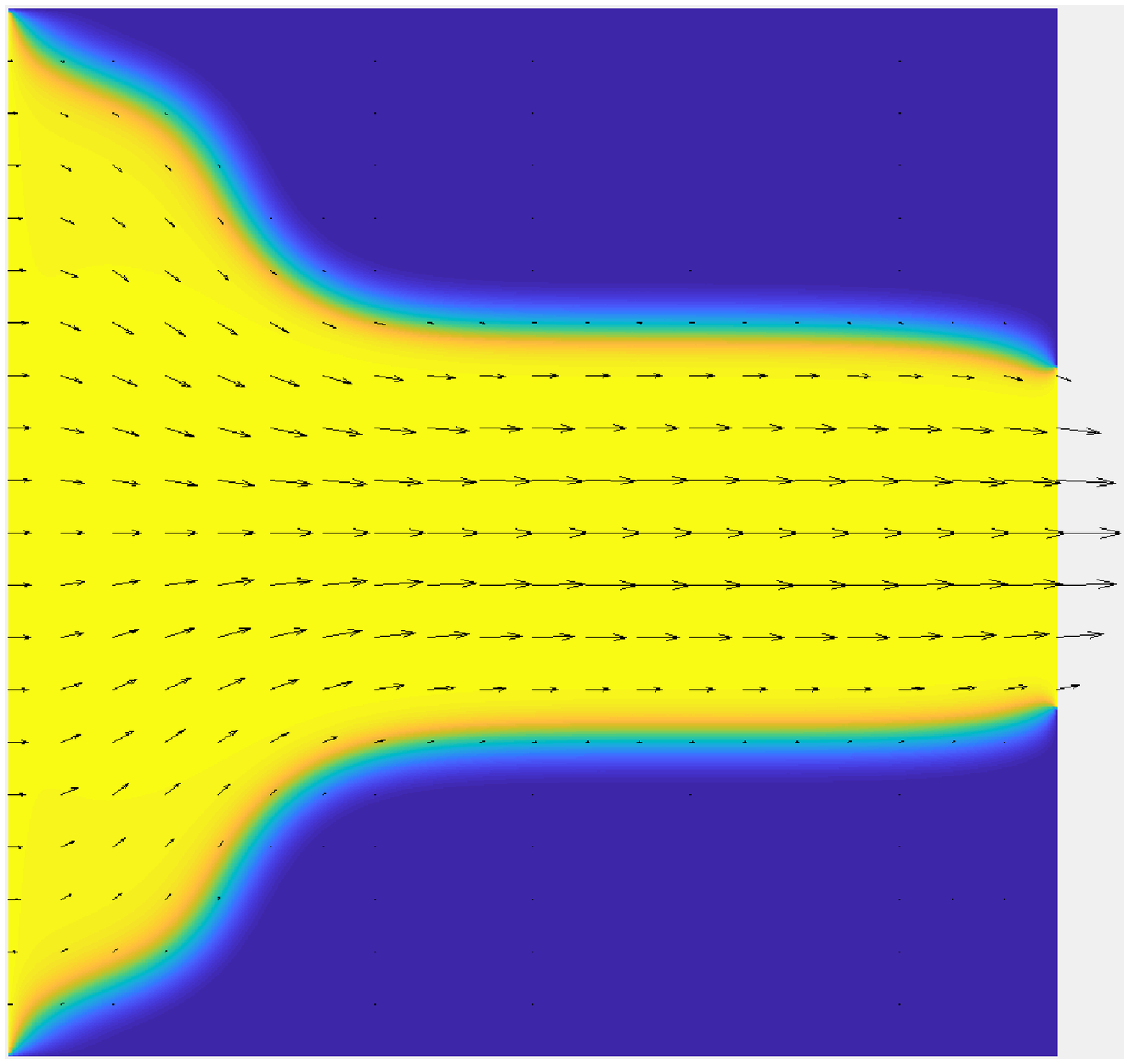}  
	\includegraphics[width=0.3\textwidth]{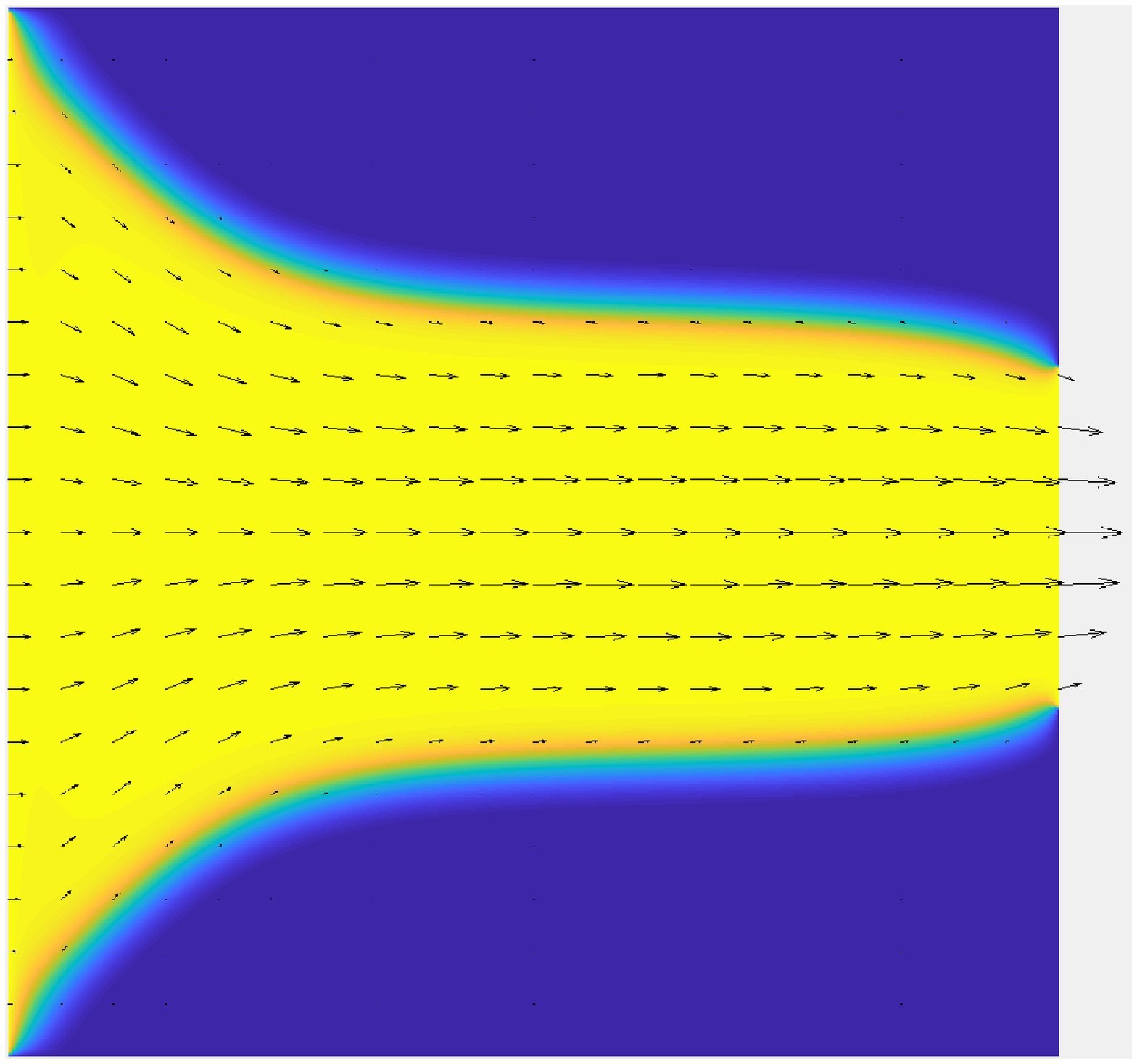}  
	\caption{Phase field functions and the fluid velocity for the diffuser at the 1st, 5th and 20th step, respectively.}
	\label{fig:diffuser-phi}
\end{figure}

Fig.\ref{fig:diffuser-obj} shows the history of objective functional and illustrate that the objective result decreases strictly. A minimum value 30.9258 appears approximately at the 20th step. 
\begin{figure}[h!]
	\centering
	\includegraphics[width=0.58\textwidth,height=3.5cm]{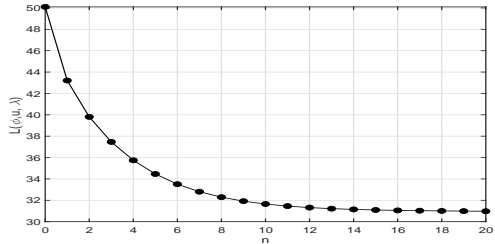} 
	\caption{Objevtive functional for the diffuser.}
	\label{fig:diffuser-obj}
\end{figure}
Furthermore, the phase-field functions of the optimized shapes are shown in Fig.\ref{fig:diffuser-phi}. The fluid velocity mainly appears in the domain of the objective domain described as $\phi=1$. Therefore, the ultimate optimized shape can be found in the interfacial layer, for instance, $\phi=0.5$. It shows no significant difference compared with the results found in Referee \cite{Borrvall03Topology}.

\subsection{Rugby}
The second example is the design of a rugby \cite{Garcke15Numerical,Borrvall03Topology,Kondoh12Drag}, whose free boundary is embedded inside. 
In this case, a computational domain is described as a rectangle $D=[0,1]\times[0,1.5]$ and the maximum flow velocity is $\bar{g} = 1$ at the inlet $\Gamma_i$ on the bottom. In our calculations, the prescribed volume fraction is kept being $\beta = 0.1189$. The initial phase-field function for rugby is given as follows,
$$
\phi^0(\mathbf{x})=
\left\{
\begin{aligned}
& 1, \ \mathbf{x} \in D \backslash \Omega_2,\\
& 0, \ \mathbf{x} \in \Omega_2,
\end{aligned}
\right.
$$
where the initial shape is similarly described as $\phi^0(\mathbf{x})=1$ in the domain $D \backslash \Omega_2$ and the void domain $\Omega_2$ is given as $\left\{(x,y)|0.3 \leq x \leq0.7, 0.2\leq y \leq 0.6\right\} $.
\begin{figure}[h!]
	\centering
	\includegraphics[width=0.3\textwidth]{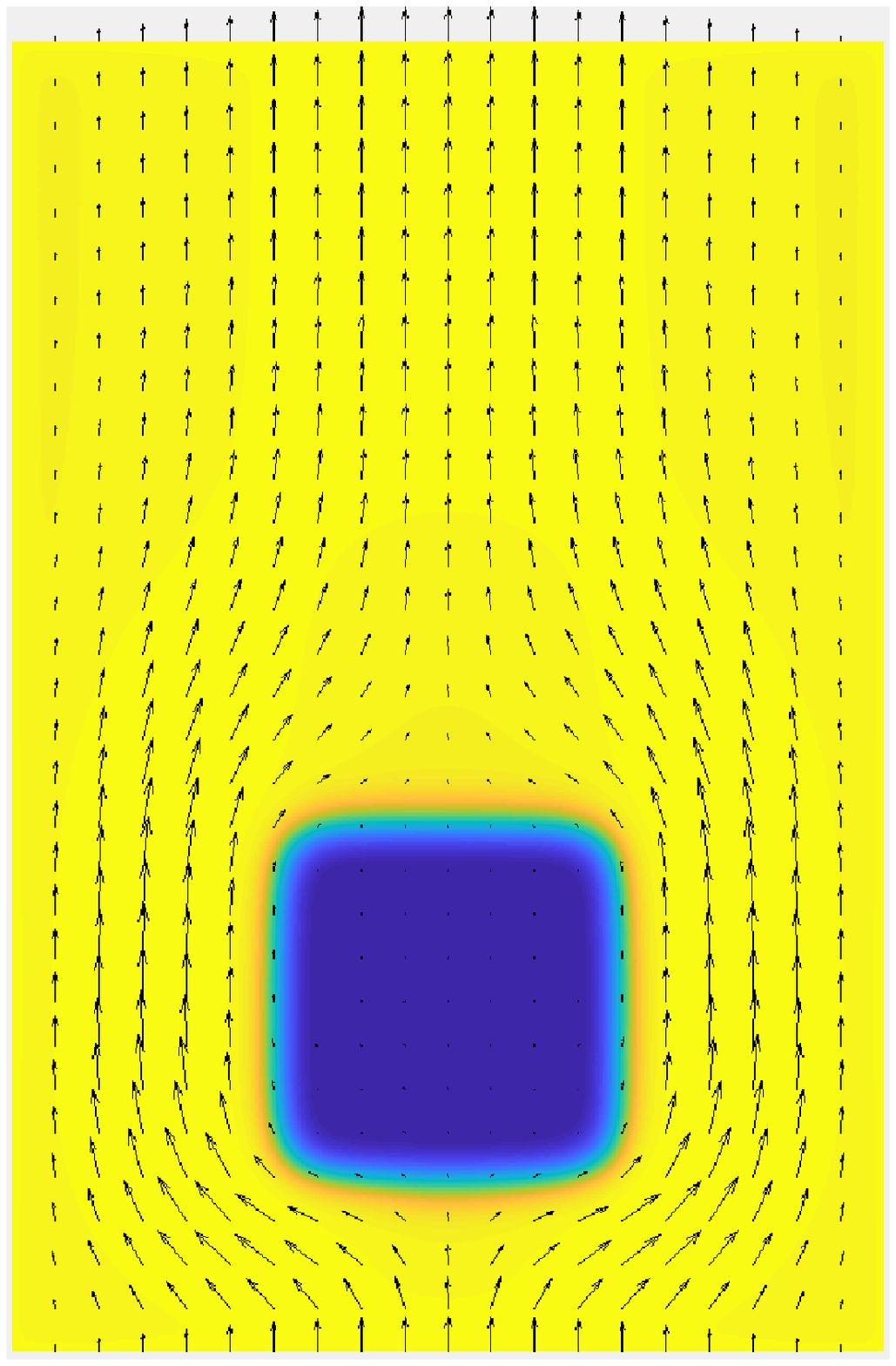} 
	\includegraphics[width=0.3\textwidth]{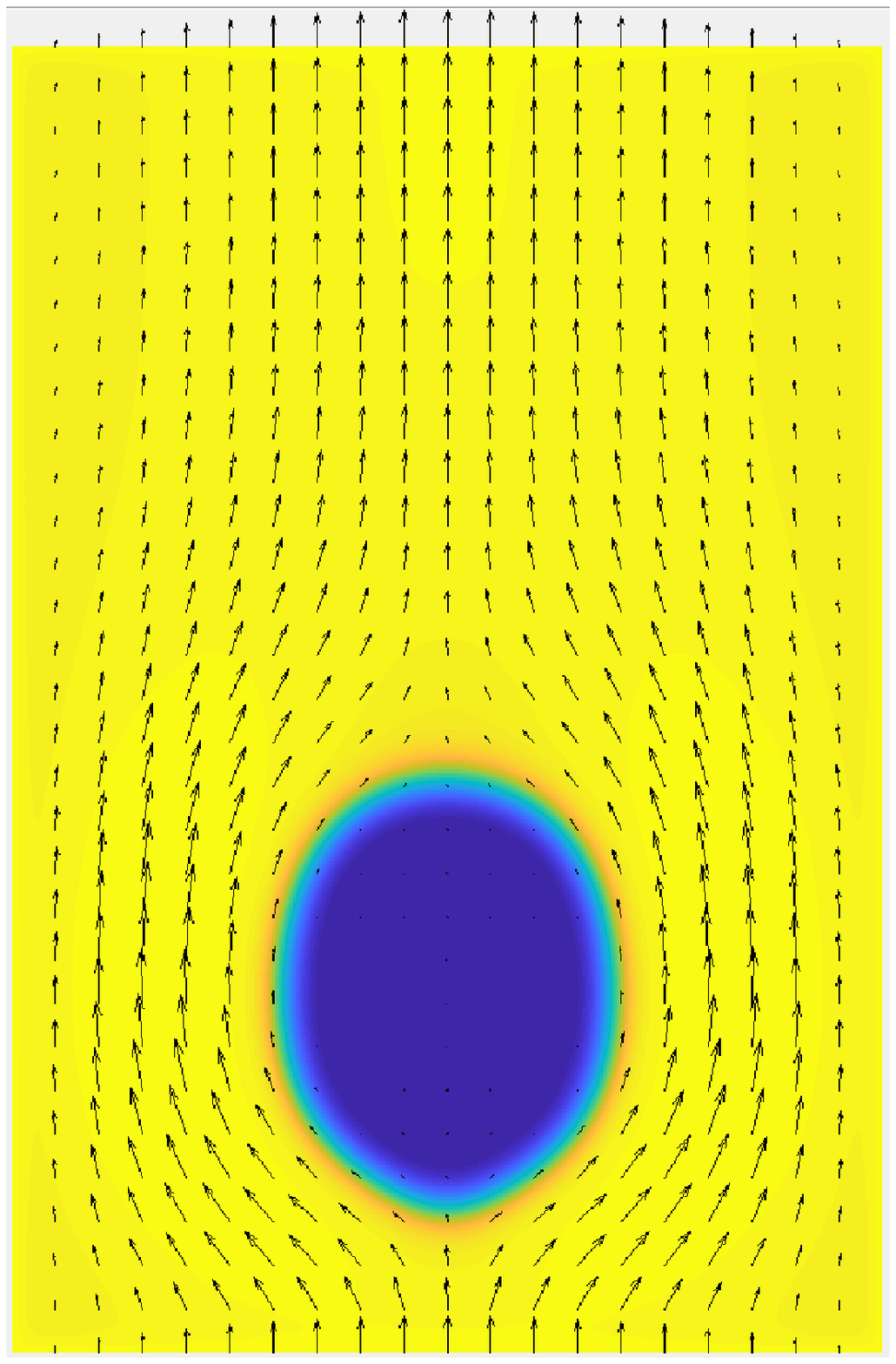}  
	\includegraphics[width=0.3\textwidth]{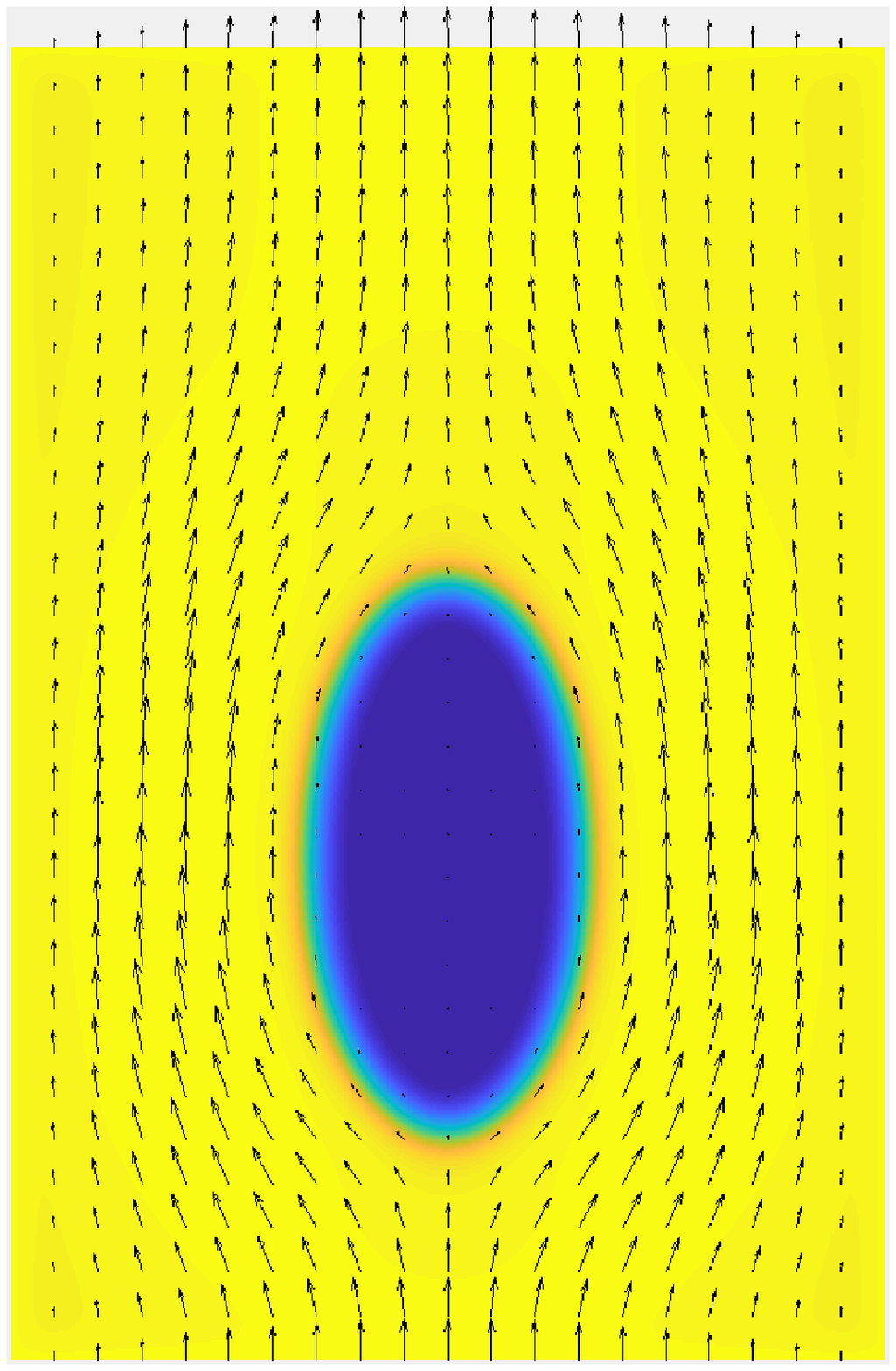}  
	\caption{Phase field functions and the fluid velocity for the rugby at the $1$st, $5$th and $50$th step, respectively.}
	\label{fig:rugby-phi}
\end{figure}

The history of objective functional is given in Fig.\ref{fig:rugby-obj} and a minimum value 25.0306 appears approximately at the 50th step, in which the objective functional decreases strictly with our designed algorithm \ref{alg:total-scheme-fully-discrete} as the time increases. We find that the initial shape of the square becomes an ellipse and the fluid goes from bottom to top away from the inner obstacle in Fig.\ref{fig:rugby-phi}.
\begin{figure}[h!]
	\centering
	\includegraphics[width=0.58\textwidth,height=3.5cm]{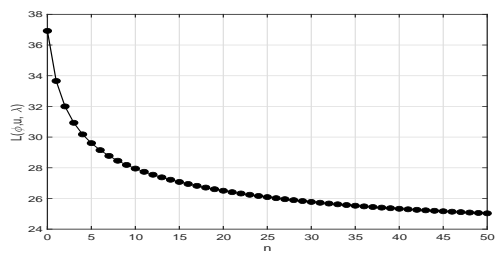} 
	\caption{Objevtive functional for the rugby.}
	\label{fig:rugby-obj}
\end{figure}

\subsection{Bypass}
We now consider a design of the arterial bypass \cite{Zhang15Topology,Abraham05Shape}, which servers as a more complex example. 
The computational domain is set to be a special domain in Fig. \ref{fig:bypass-domain} and the maximum flow velocity is $\bar{g} = 1$ at the inlet $\Gamma_i$ on the left. In our calculations, the prescribed volume fraction is set to be $\beta = 0.2777$. In this case, we adopt the random numbers in $[0,1]$ for the phase-field functions such that the volume is kept. After several iterations, the distributions of the phase-field variable are used for the initial phase-field function. 
\begin{figure}[h!]
	\centering
	\includegraphics[width=0.5\textwidth]{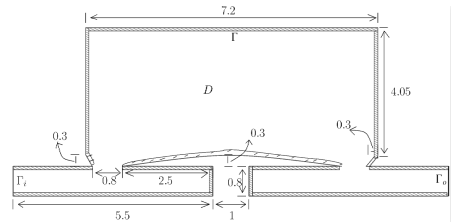}  
	\caption{The computational domain for the bypass.}
	\label{fig:bypass-domain}
\end{figure}
\begin{figure}[h!]
	\centering
	\subfigure[$1$st step]{
		\includegraphics[width=0.4\linewidth]{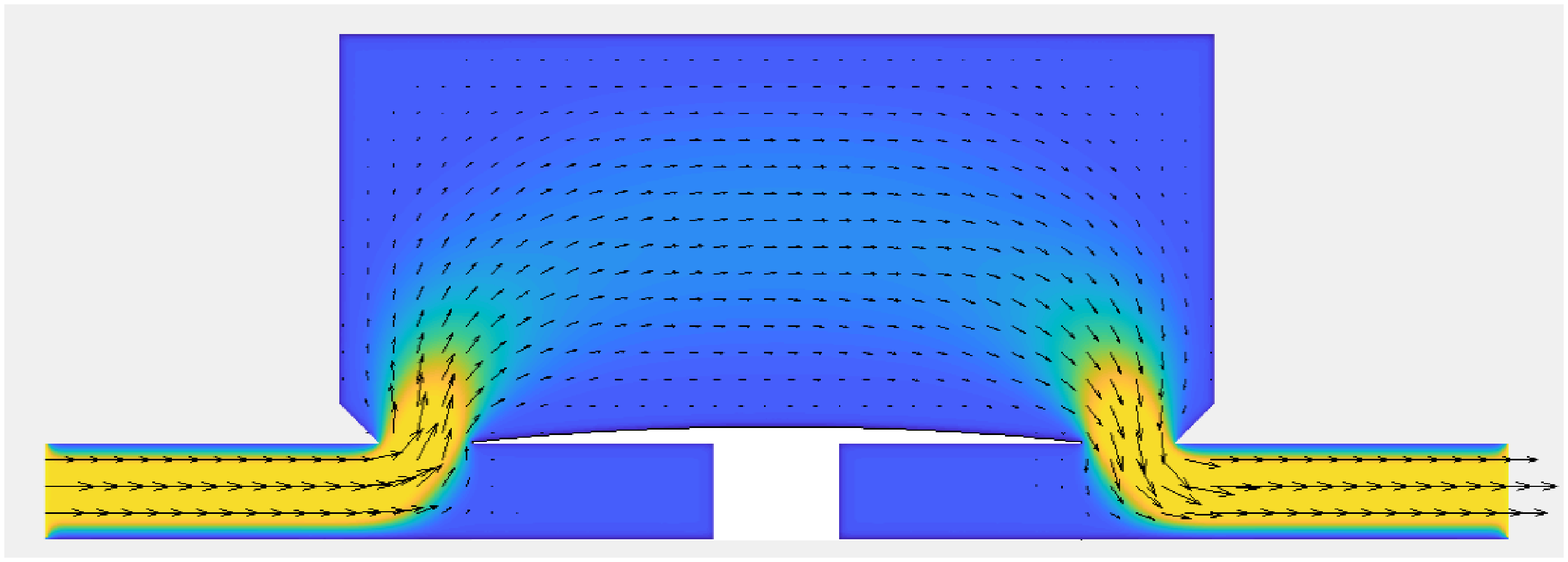}}
	\subfigure[$5$th step]{
		\includegraphics[width=0.4\linewidth]{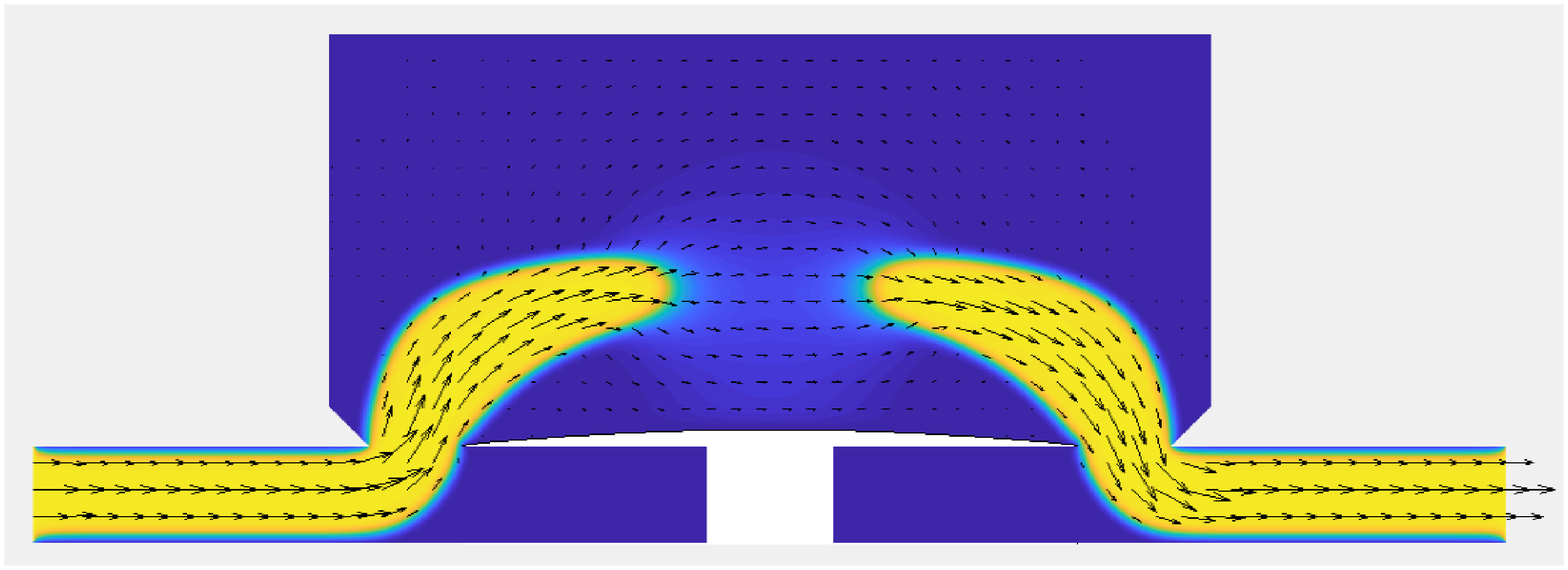}}
	\subfigure[$10$th step]{
		\includegraphics[width=0.4\linewidth]{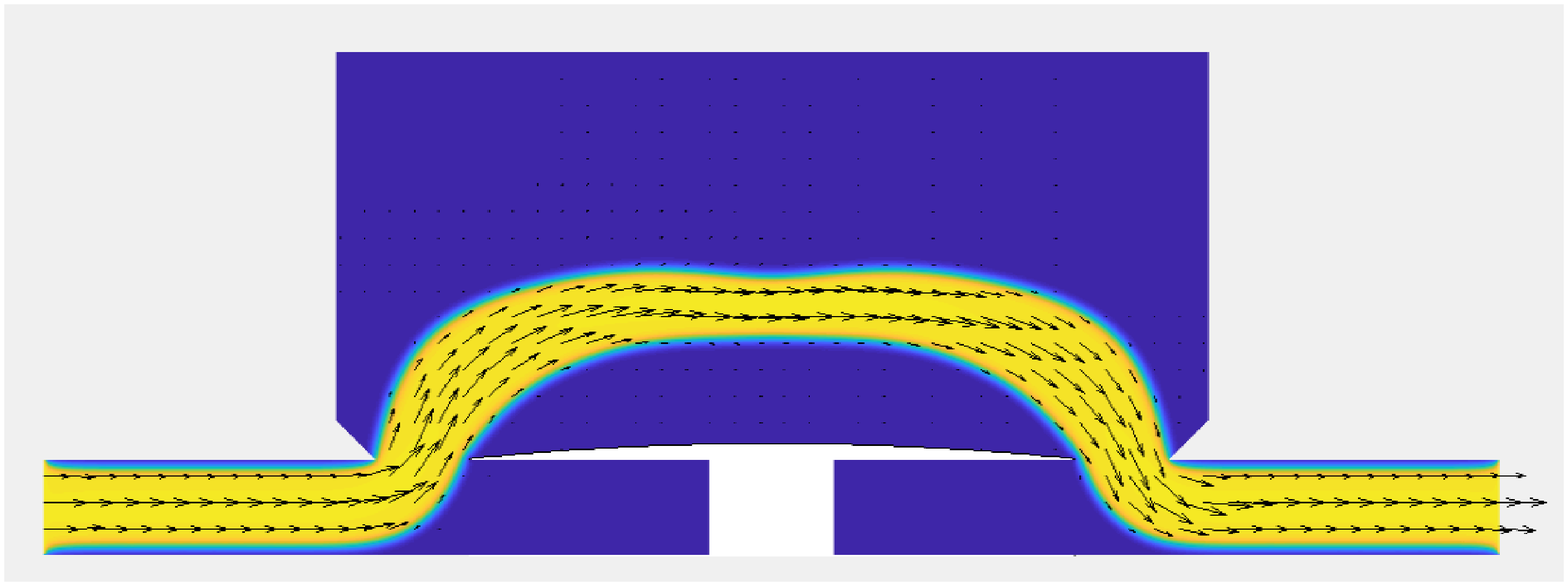}}
	\subfigure[$50$th step]{
		\includegraphics[width=0.4\linewidth]{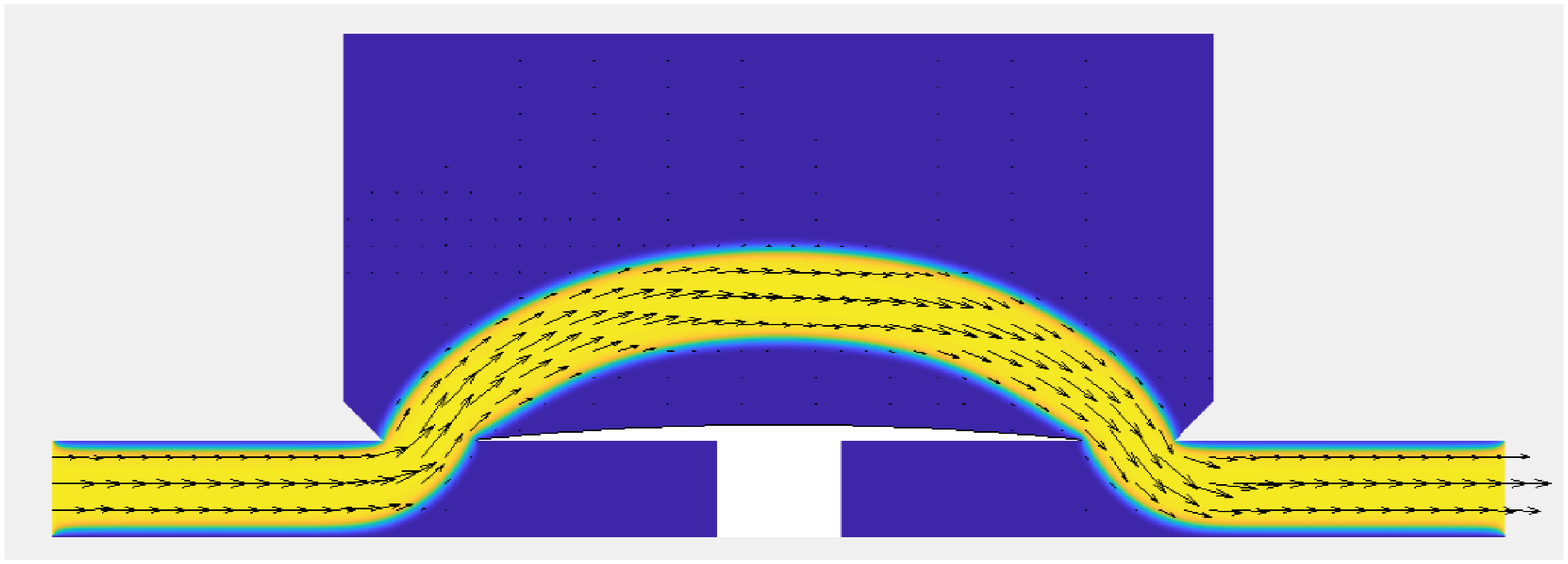}}
	\caption{Phase field functions and the fluid velocity for the bypass.}
	\label{fig:bypass-phi}
\end{figure}

Fig.\ref{fig:bypass-obj} shows the history of objective functional and a minimum value 76.1621  reaches approximately at the 50th step. As time goes, we find that the shapes of arterial vessels form after about 10th step in Fig.\ref{fig:bypass-phi}. After 50th step, the ultimate shape appears, which is similar to the results in \cite{Abraham05Shape,Zhang15Topology}.
\begin{figure}[h!]
	\centering
	\includegraphics[width=0.58\textwidth,height=3.5cm]{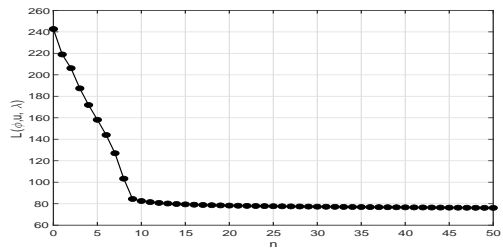} 
	\caption{Objevtive functional for the bypass.}
	\label{fig:bypass-obj}
\end{figure}

\subsection{Pipe bend} Next, we consider an example on the design of a pipe bend. 
The computational domain is still a square $D=[0,1]^2$ and the prescribed volume fraction is $\beta=0.2762$. Differ from the foregoing examples, the initial shape is designed as the multiple disconnected regions. 
What's more, the initial phase-field function for the pipe bend is described as,
$$
\phi^0(\mathbf{x})=
\left\{
\begin{aligned}
& 1, \ \mathbf{x} \in D \backslash \Omega_4,\\
& 0, \ \mathbf{x} \in \Omega_4,
\end{aligned}
\right.
$$
where the initial shape is given as $\phi^0(\mathbf{x})=1$ in the domain $D \backslash \Omega_4$ and the void domain $\Omega_4$ contains the sixteen similar circles with the same radius $r=0.12$. The initial phase-field function and the corresponding velocity can be found in Fig.\ref{fig:Pipebend-phi}(a).
\begin{figure}[h!]
	\centering
	\subfigure[Initial]{
		\includegraphics[width=0.3\linewidth]{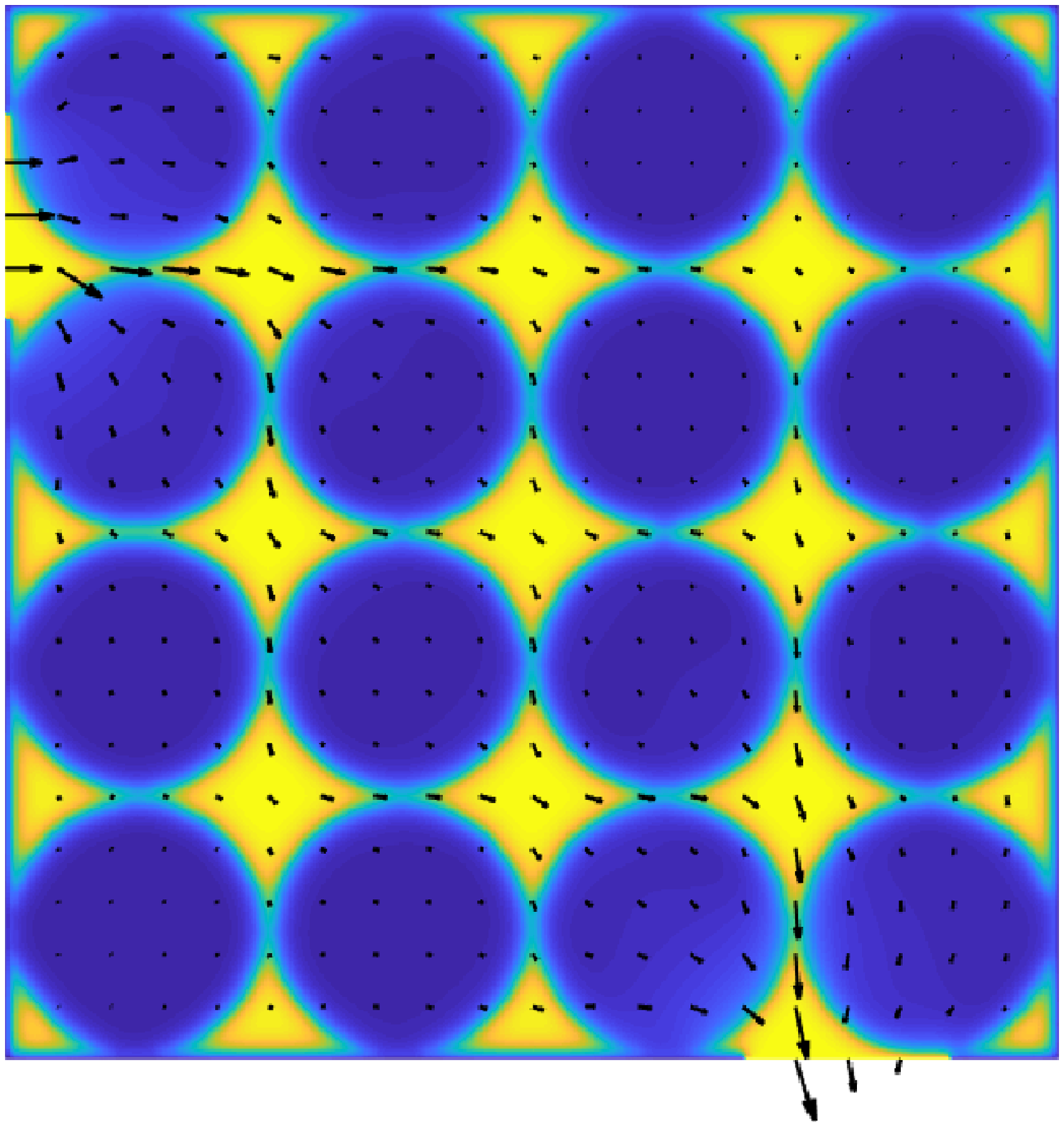}}
	\subfigure[$1$st step]{
		\includegraphics[width=0.3\linewidth]{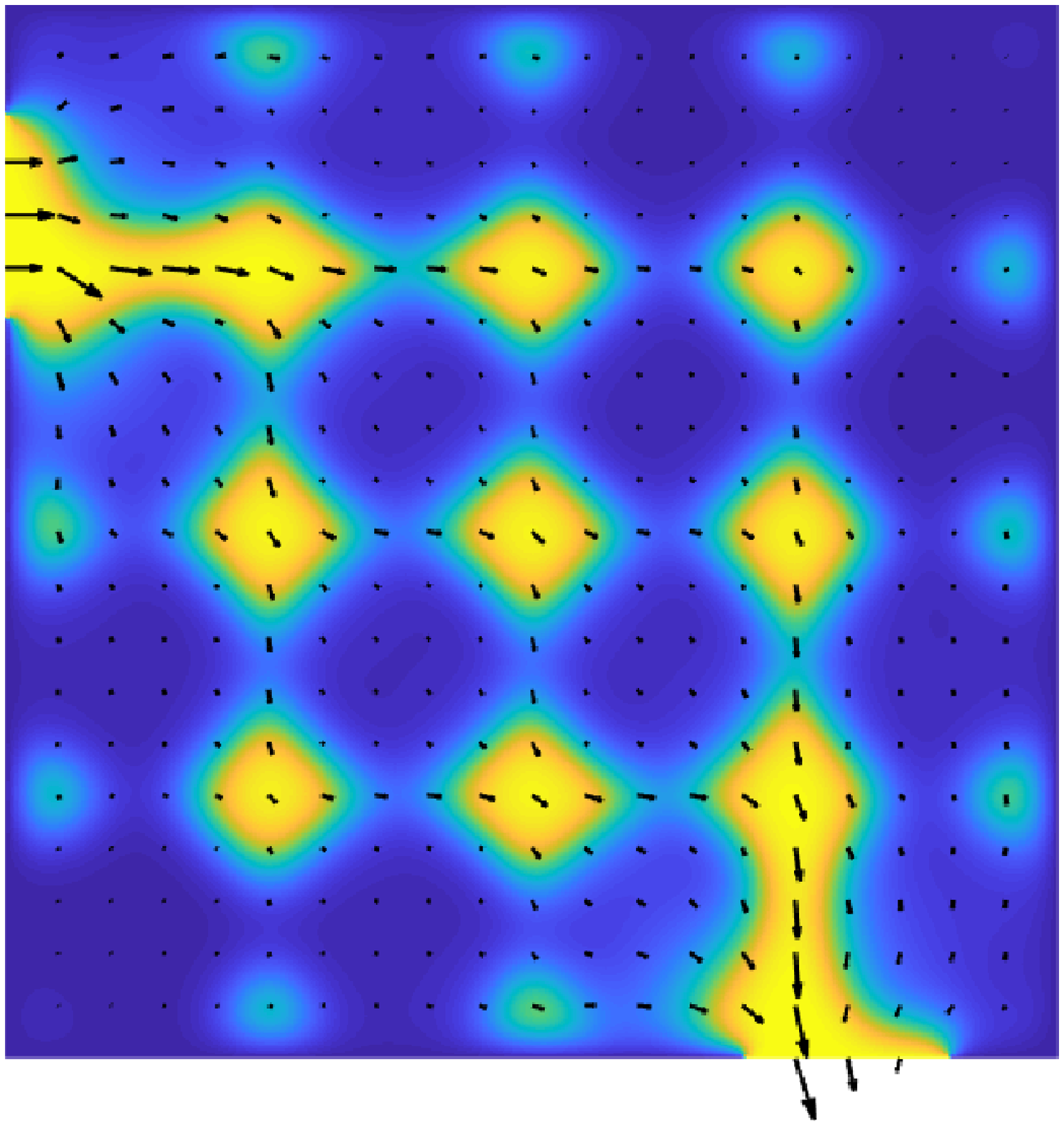}}
	\subfigure[$2$ed step]{
		\includegraphics[width=0.3\linewidth]{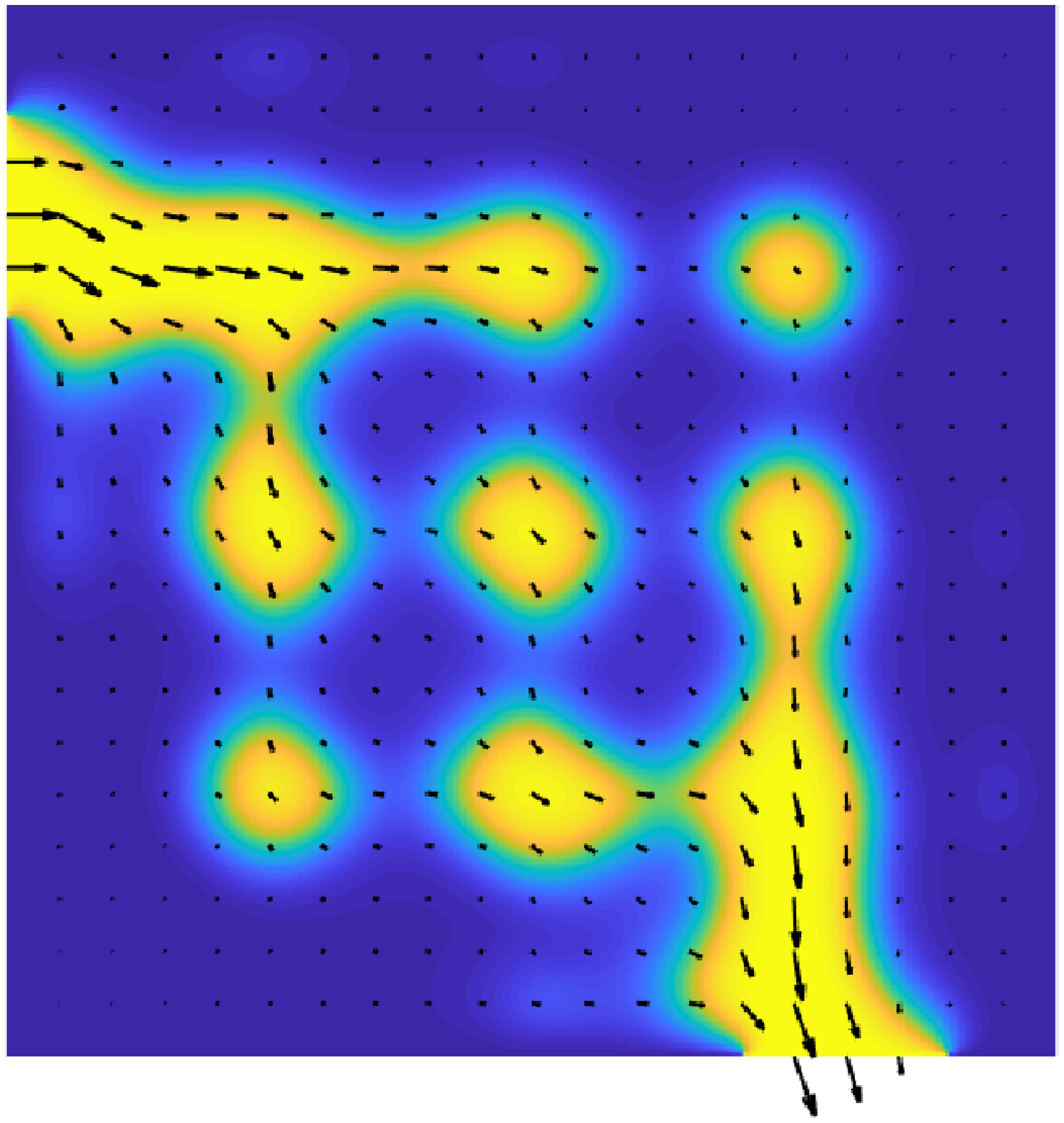}}\\
	\subfigure[$4$th step]{
		\includegraphics[width=0.3\linewidth]{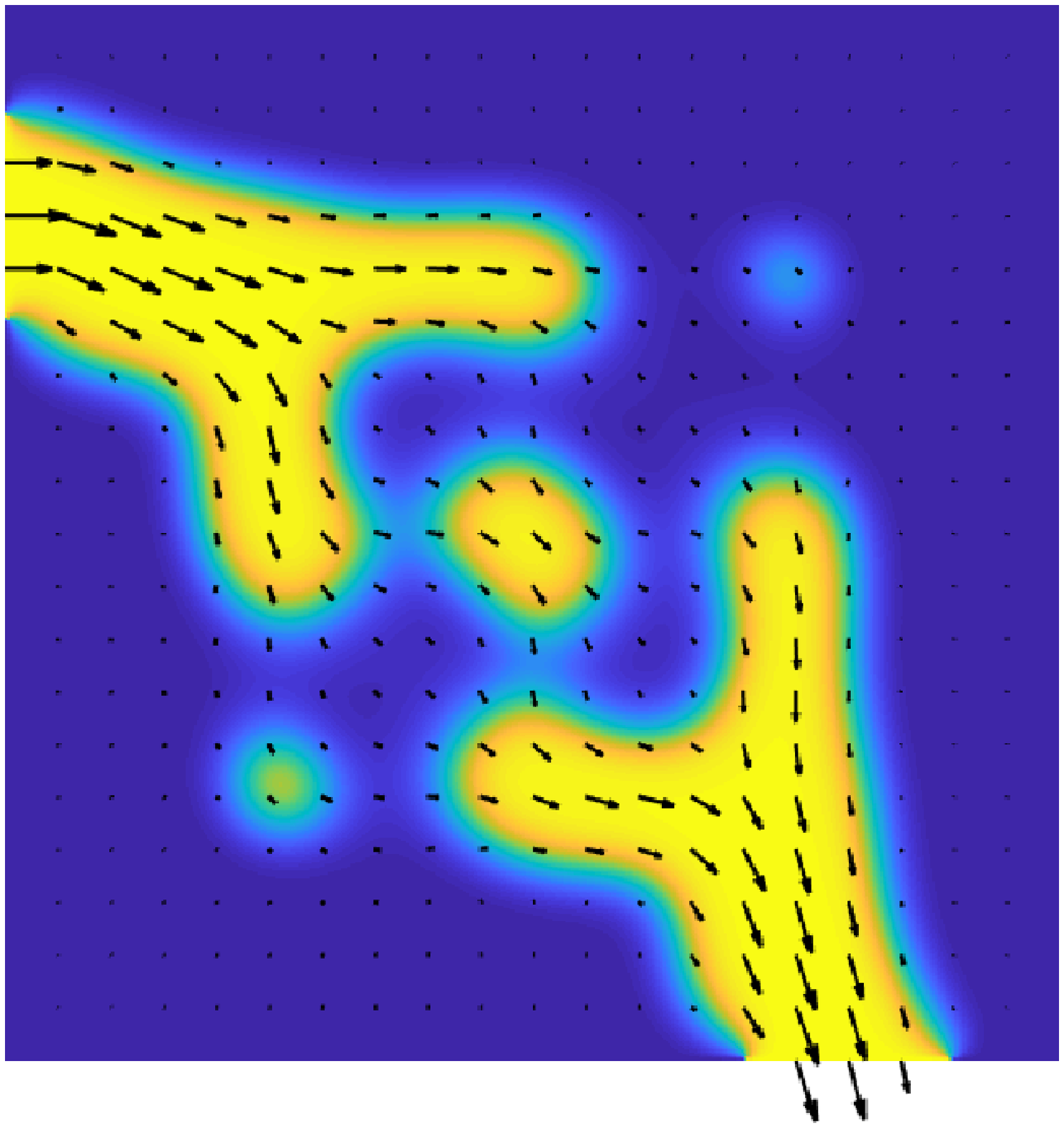}}
	\subfigure[$6$th step]{
		\includegraphics[width=0.3\linewidth]{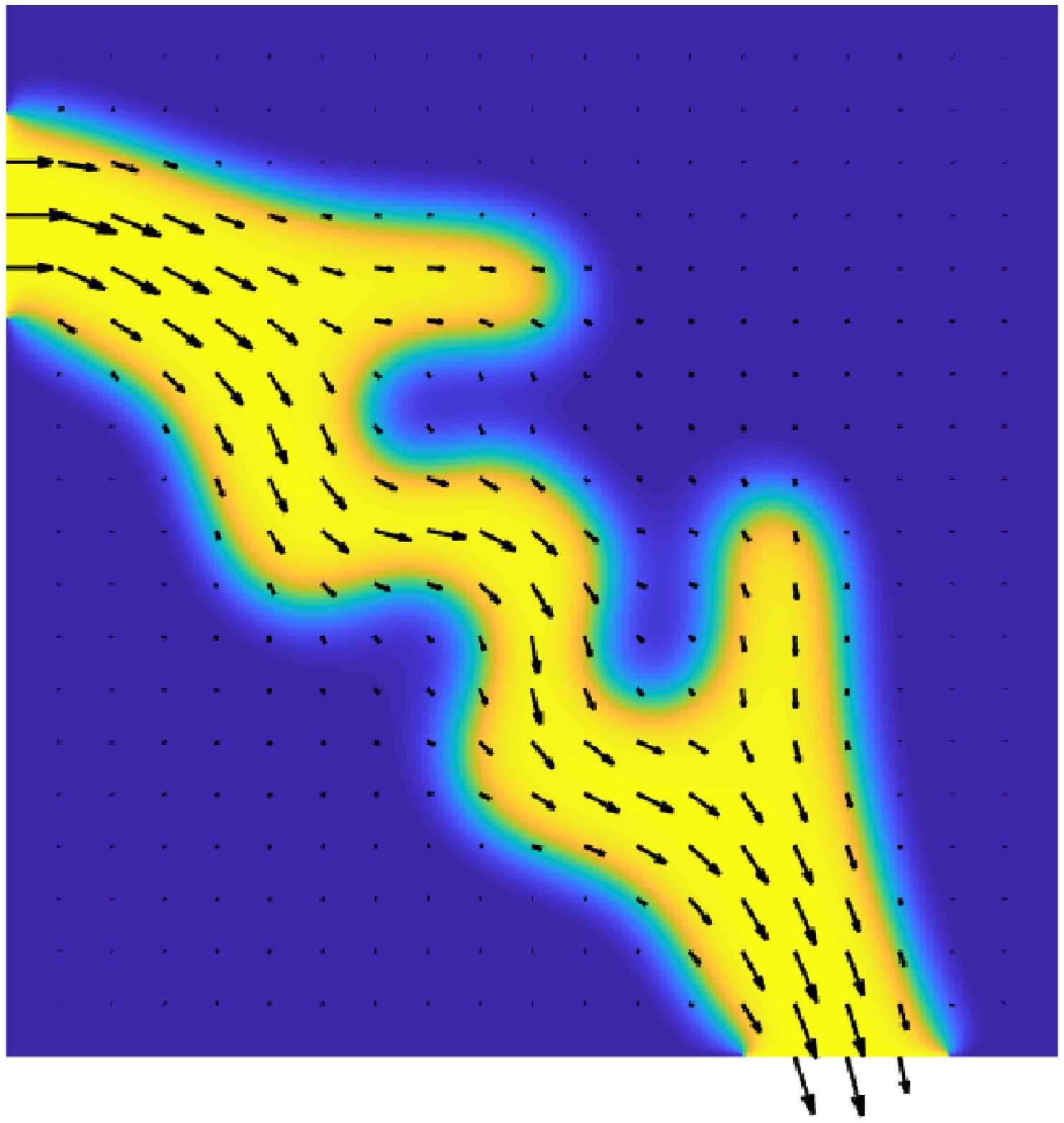}}
	\subfigure[$20$th step]{
		\includegraphics[width=0.3\linewidth]{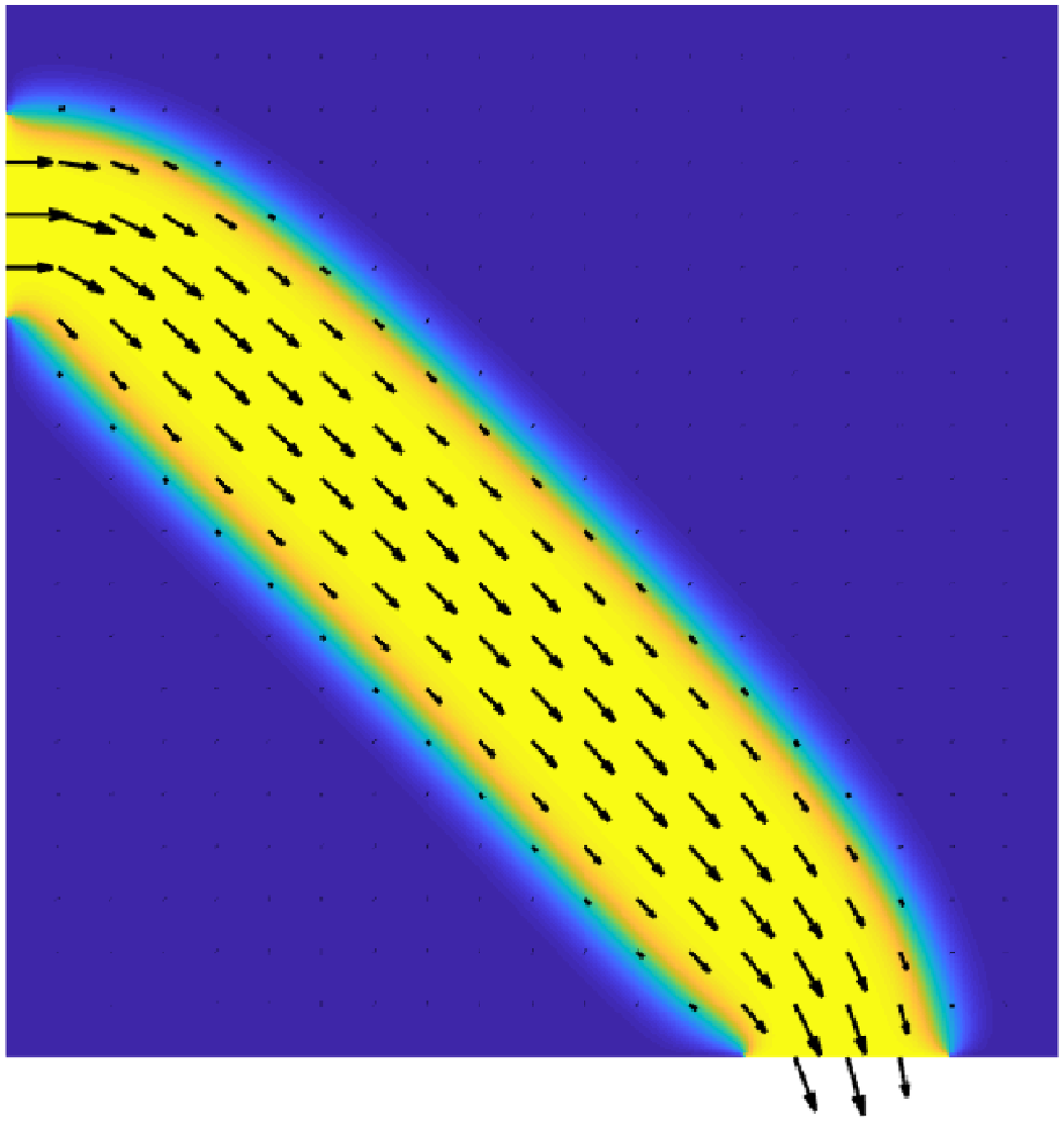}}
	\caption{Phase field functions and the fluid velocity for the pipe bend.}
	\label{fig:Pipebend-phi}
\end{figure}

Fig.\ref{fig:pipebend-obj} shows the history of objective functional and a minimum value 8.7916 appears approximately at the $20$th step, where the objective functional decreases strictly as time increases. Fig.\ref{fig:Pipebend-phi} shows the results of evolution on the phase-field functions and fluid velocity. We find that the topological structure has been changed as time goes, and ultimately forms a pipe bend.
\begin{figure}[h!]
	\centering
	\includegraphics[width=0.58\textwidth,height=3.5cm]{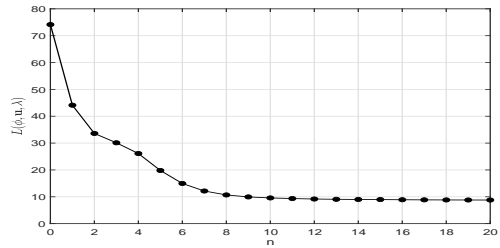} 
	\caption{Objevtive functional for the pipe bend.}
	\label{fig:pipebend-obj}
\end{figure}

\subsection{Diffuser in 3D} We now consider the design of a diffuser in 3D \cite{ChenAn}.
The computational domain is set to be $(D=[0,1]^3)$ and the maximum flow velocity is $\bar{g} = 1$ at the inlet $\Gamma_i$ on the left. 
In our calculations, the prescribed volume fraction is set to be $\beta = 0.328$. 
The initial phase-field function for the diffuser in 3D is described as,
$$
\phi^0(\mathbf{x})=
\left\{
\begin{aligned}
	& 1, \ \mathbf{x} \in \Omega_5,\\
	& 0, \ \mathbf{x} \in D \backslash \Omega_5,
\end{aligned}
\right.
$$
where we set the domain $\Omega_5:=\left\{(x,y)| x \leq0.3\} \cup  \{(x,y)| 0.4 \leq y \leq 0.6, 0.4 \leq z \leq 0.6\right\}$ and give the initial shape that is described as $\phi^0(\mathbf{x})=1$ in the domain $\Omega_5$. 
\begin{figure}[h!]
	\centering
	\subfigure[$\phi=0.5$ at the 1st step]{
		\includegraphics[width=0.30\linewidth]{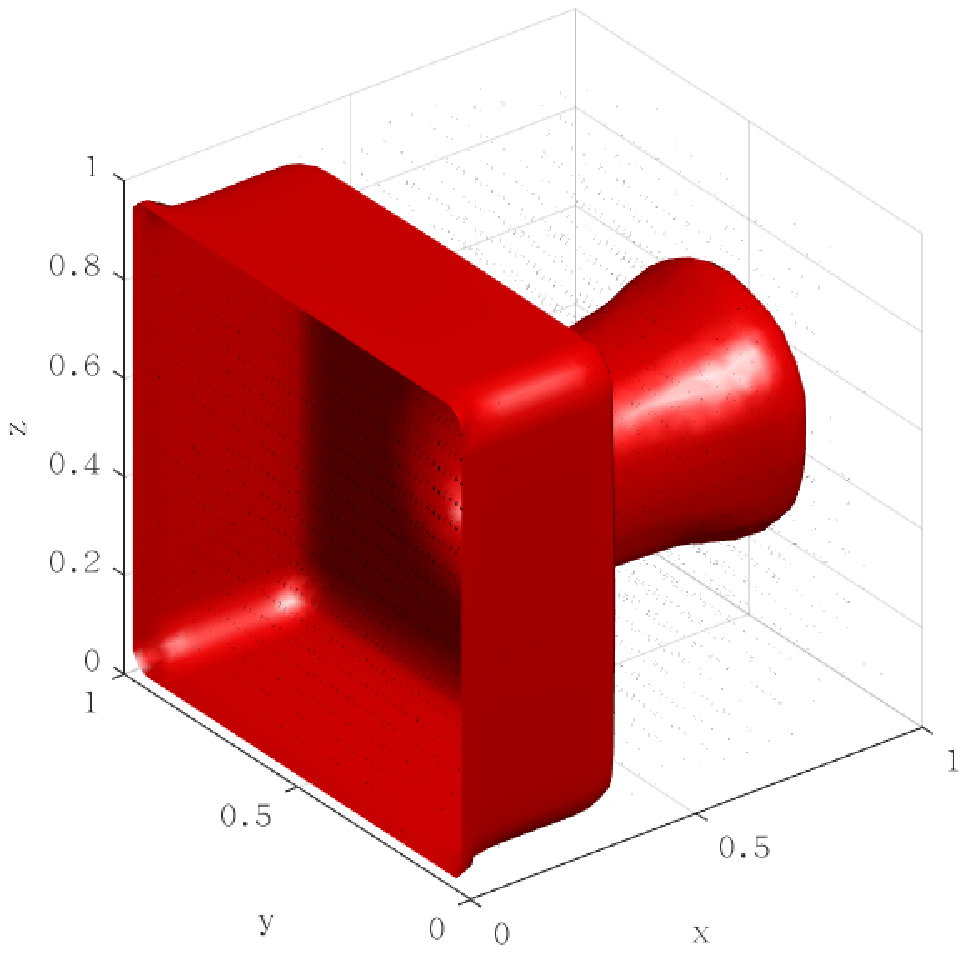}}
	\subfigure[$\phi=0.5$ at the 5th step]{
		\includegraphics[width=0.30\linewidth]{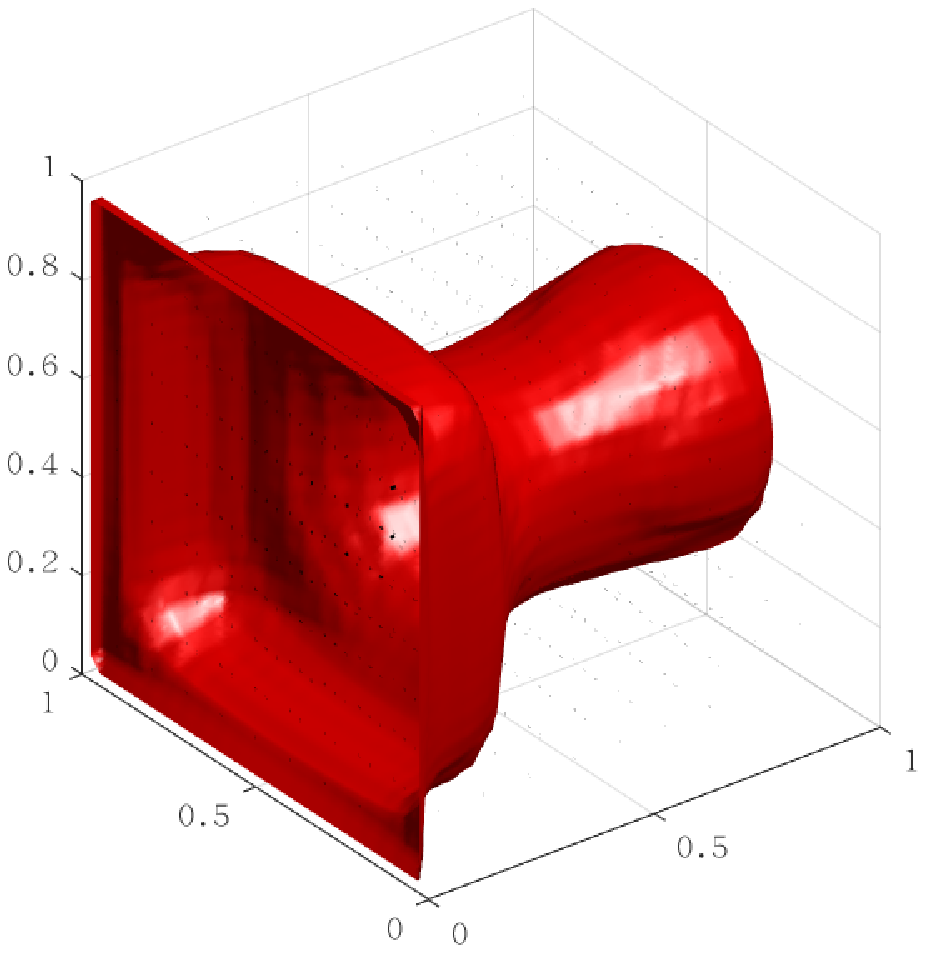}}
	\subfigure[$\phi=0.5$ at the 30th step]{
		\includegraphics[width=0.30\linewidth]{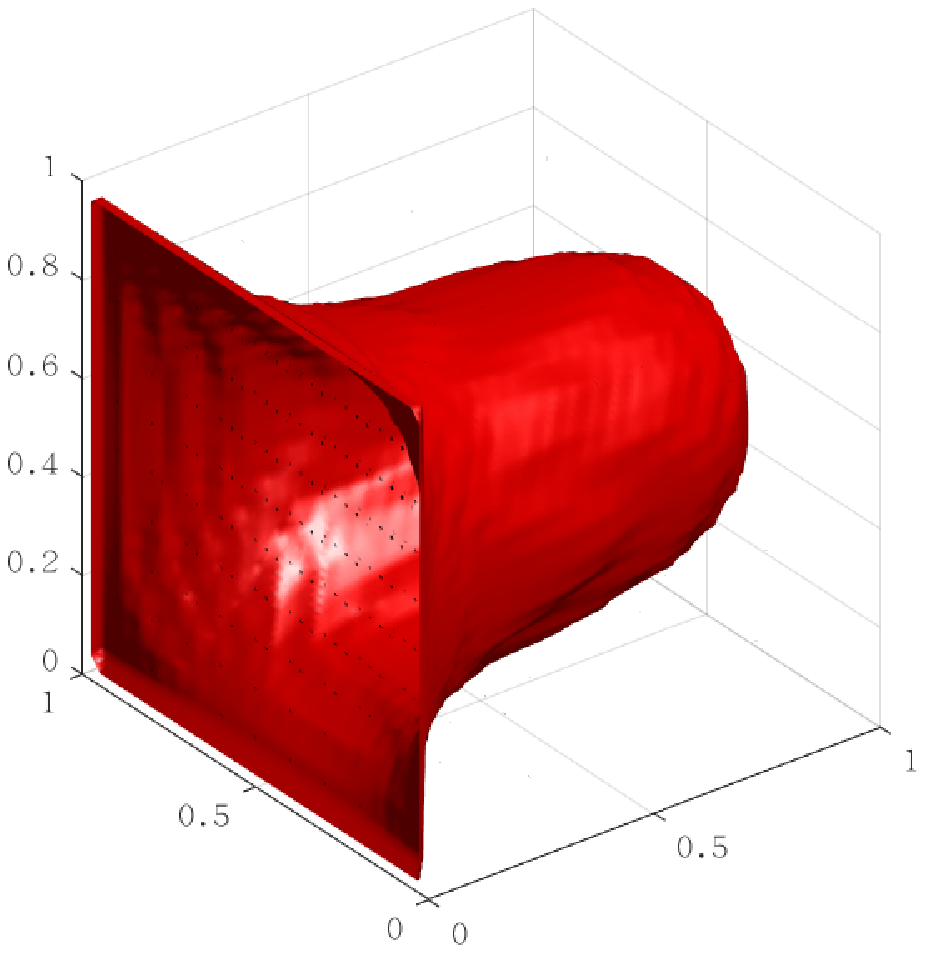}} \\
	\subfigure[slice at the 1st step]{
		\includegraphics[width=0.30\linewidth]{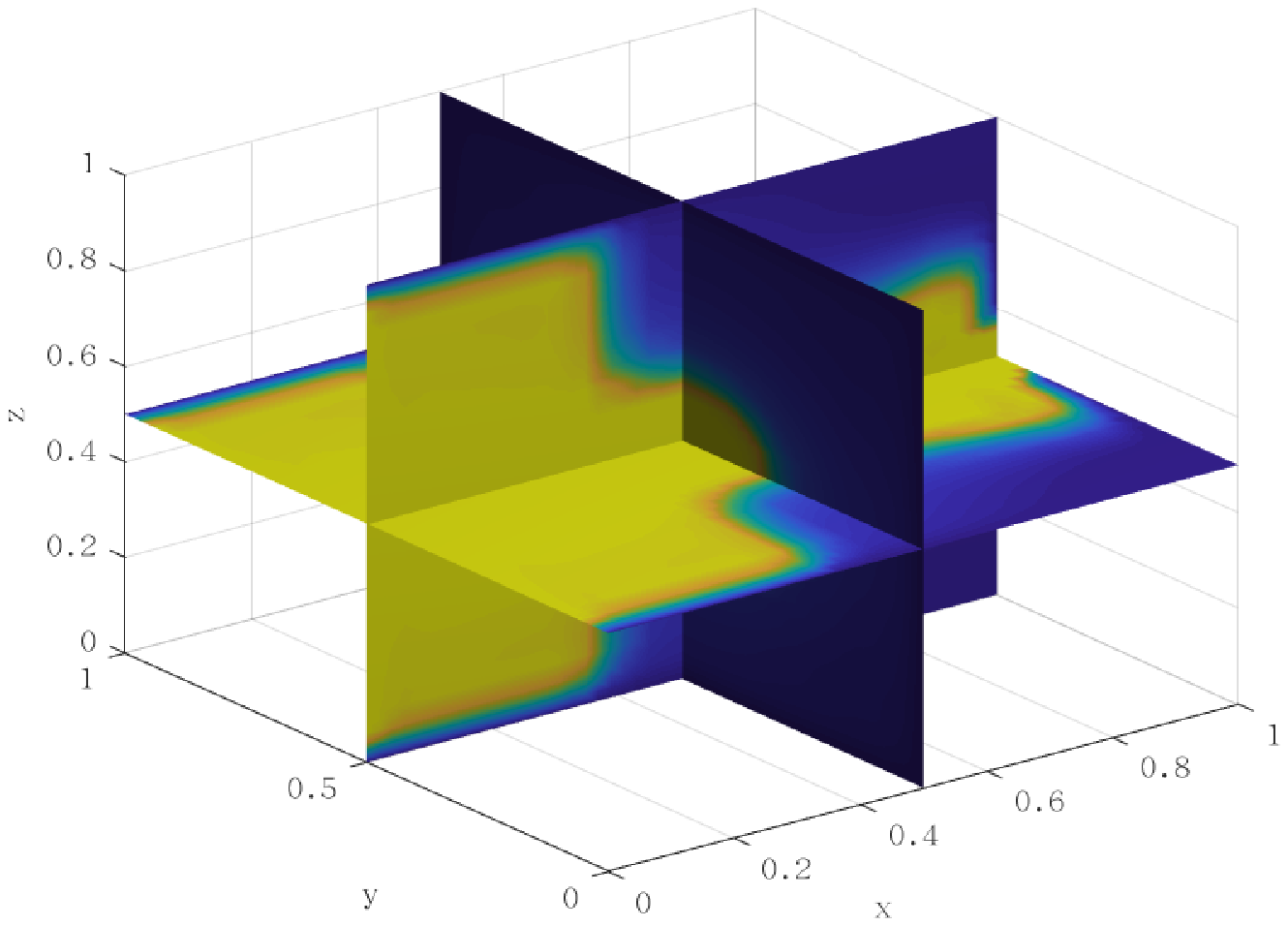}}
	\subfigure[slice at the 5th step]{
		\includegraphics[width=0.30\linewidth]{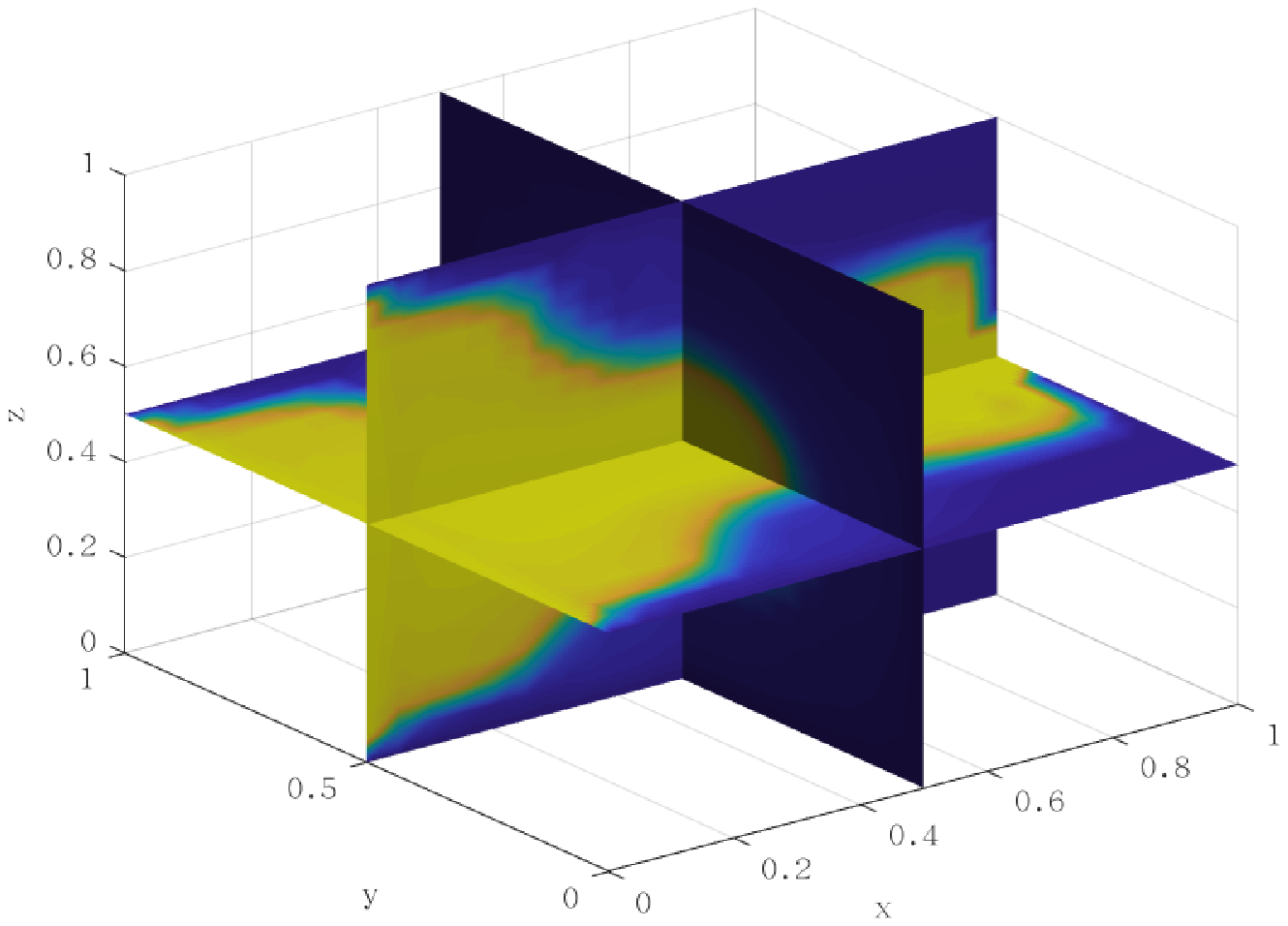}}
	\subfigure[slice at the 30th step]{
		\includegraphics[width=0.30\linewidth]{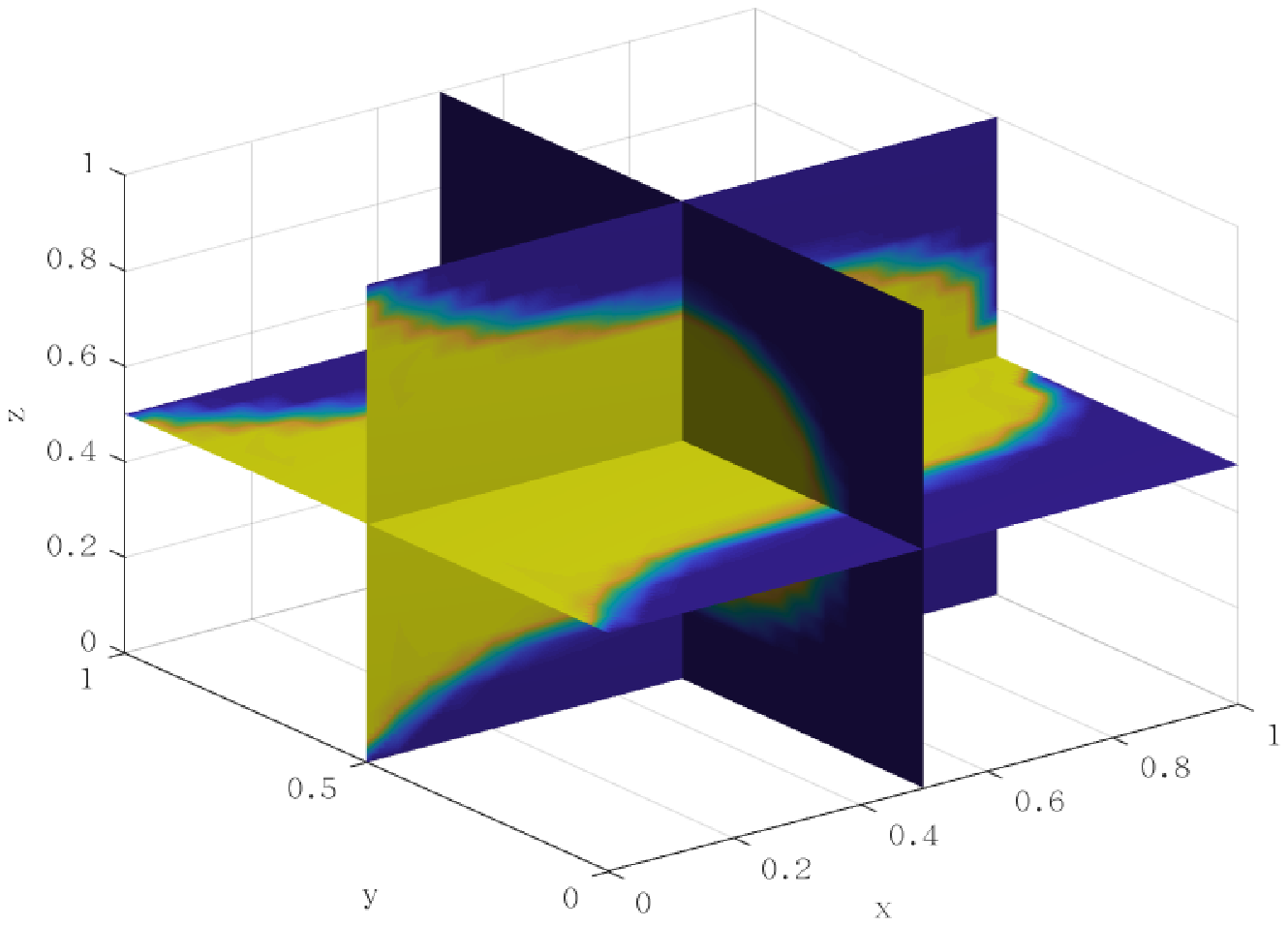}}
	\caption{Phase field functions $\phi=0.5$ for the diffuser in 3D in the first row; slices ($x=0.5,y=0.5,z=0.5$) of the optimal design results in the second row.}
	\label{fig:diffuser3D-phi}
\end{figure}
\begin{figure}[h!]
	\centering
	\includegraphics[width=0.58\textwidth,height=3.5cm]{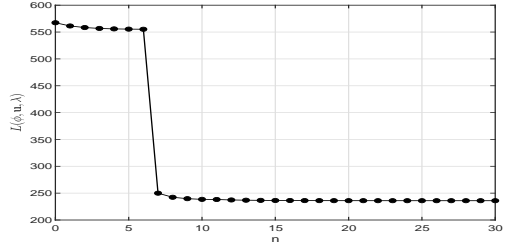} 
	\caption{Objevtive functional for the diffuser in 3D.}
	\label{fig:diffuser3D-obj}
\end{figure}

Fig.\ref{fig:diffuser3D-obj} shows the history of objective functional and it approximately reaches a steady state around the 30th step. Obviously, for the energy functional $L(\phi,\mathbf{u},\lambda)$, it has the monotonic-decaying property, which also matches the forgoing theoretical result.
Furthermore, we provide the isolines and the slices of optimal design result in Fig.\ref{fig:diffuser3D-phi}. The optimal design results seem to be similar to that in \cite{ChenAn}.

\section{Conclusions}
\label{sec:conclusions}
We have proposed an efficient decoupled monotonic-decaying algorithm for shape optimization in Stokes flows with the phase-field method, combining a cut-off postprocessing technique of the phase-field function and a linear piecewise finite element approximation for the phase-field variable. Moreover, we rigorously proved that this decoupled scheme can decrease the objective function step by step. Numerical examples in 2D and 3D are carried out to validate the effectiveness of our proposed numerical scheme.

The algorithm is also directly applicable to the structure optimization with minimizing the compliance \cite{Takezawa10Shape}. But the scope of the present article still has its limitation: it requires a strong relevance between the variation of objective functional with respect to the state variable and the state equation, and it is nontrivial to extend such decoupled monotonic-decaying algorithm for general optimal control problems. Hence, it would be a meaningful and challenge future work.


\section*{\bf Acknowledgements}

This work is partially supported by  the National Natural Science Foundation of China (NSFC) Grant No. 11871264, the NSFC/Hong Kong RGC Joint Research Scheme (NSFC/RGC 11961160718), and the fund of the Guangdong Provincial Key Laboratory of Computational Science and Material Design (No. 2019B030301001). 

\bibliographystyle{siamplain}
\bibliography{references}
\end{document}